\documentclass[11pt]{article}
\usepackage{amssymb,latexsym,amsmath,amsbsy,amsthm,amsxtra,amsgen,amssymb,dsfont,pdfsync,color,graphicx}
\newtheorem{Theorem}{Theorem}
\newtheorem{Lemma}{Lemma}
\newtheorem{Proposition}{Proposition}

\DeclareFontFamily{U}{mathx}{\hyphenchar\font45}
\DeclareFontShape{U}{mathx}{m}{n}{
     <5> <6> <7> <8> <9> <10>
     <10.95> <12> <14.4> <17.28> <20.74> <24.88>
     mathx10
     }{}
\DeclareSymbolFont{mathx}{U}{mathx}{m}{n}
\DeclareFontSubstitution{U}{mathx}{m}{n}
\DeclareMathAccent{\widecheck}{9}{mathx}{"71}

\begin{document}
\title{Tree lengths  for general $\Lambda$-coalescents and the asymptotic site frequency spectrum around the Bolthausen-Sznitman coalescent}

\author{Christina S. Diehl\thanks{Institut f\"ur Mathematik, Goethe Universit\"at, Frankfurt am Main, Germany \newline diehl@math.uni-frankfurt.de, kersting@math.uni-frankfurt.de \newline Work partially supported by the DFG Priority Programme SPP 1590 ``Probabilistic Structures in Evolution''} $\ $ and G\"otz Kersting$^*$}
\maketitle
\begin{abstract}
We study tree lengths in  $\Lambda$-coalescents without a dust component from a sample of $n$ individuals. For the total length of all branches  and the total length of all external branches we present laws of large numbers in full generality. The other results treat regularly varying coalescents with exponent 1, which cover the Bolthausen-Sznitman coalescent. The theorems contain  laws of large numbers for the total length of all internal branches and of internal branches of order $a$ (i.e. branches carrying $a$ individuals out of the sample). These results transform immediately to sampling formulas in the infinite sites model. In particular, we obtain the asymptotic site frequency spectrum of the Bolthausen-Sznitman coalescent. The proofs rely on a new technique to obtain laws of large numbers for certain functionals of decreasing Markov chains.
\end{abstract}

\section{Introduction and main results}
$\Lambda$-coalescents are established models for family trees of a sample of individuals from some large population. Its most prominent representative, the Kingman coalescent \cite{Ki}, is widely used in population genetics. More recently the Bolthausen-Sznitman coalescent \cite{Bo} gained attention for models including selection. The  class of Beta-coalescents with parameter $1<\alpha < 2$ has been applied to marine populations \cite{Bi}. In this paper we investigate branch lengths and sampling formulas in the infinite sites model for the whole class of $\Lambda$-coalescents without a dust component, which cover all these special cases.

  $\Lambda$-coalescents have been introduced by Pitman \cite{Pi} and Sagitov \cite{Sa} as  Markov processes whose states are partitions of $\mathbb N$ and whose evolution may be imagined as a random tree. In this paper we identify a $\Lambda$-coalescent  with an induced sequence of $n$-coalescents, $n \in \mathbb N$, by restricting  partitions to the subsets $\{1,\ldots,n\}$ of $\mathbb N$. These $n$-coalescents are considered to be  Markovian models for the family trees of a sample of $n$ individuals. If such a tree contains $b \in \{2,\ldots,n\}$ lineages at the moment $t \ge 0$ backwards in time (representing the ancestors living at that moment), then it is assumed that out of them $k\in \{2,\ldots,b\}$ specified lines merge at rate
\[ \lambda_{b,k} := \int_{[0,1]} p^k(1-p)^{b-k} \frac{\Lambda(dp)}{p^2} \]
to one line. Here $\Lambda$ denotes any finite, non-vanishing measure on $[0,1]$. The resulting tree consists of $n$ leaves, $\tau_n$ merging events, and a root at time $\tilde \tau_n >0$ representing the MCRA (most recent common ancestor). Its branches have lengths which specify  lifetimes. There are external branches ending in leaves on the one hand and internal branches ending in mergers on the other. For the  detailed partition valued picture we refer to the survey \cite{Be1}. 

In the sequel we  work with the Markovian  counting process $N_n=(N_n(t))_{t \ge 0}$, where $N_n(t)$ denotes the number of lineages at time $t\ge 0$. Thus $N_n(0)=n$ and $N_n(\tilde \tau_n)=1$. For convenience we set $N_n(t):=1$ for $t >\tilde \tau_n$. Then for $b=2, \ldots,n$ the numbers
\[\lambda(b)= \sum_{k=2}^b \binom bk \lambda_{b,k}   \]
give the jump rates of the process $N_n$ and 
\[ \mu(b)= \sum_{k=2}^b (k-1)\binom bk \lambda_{b,k} \]
its rate of decrease, since  a merger of $k$ lineages results in a downward jump of the block counting process of size $k-1$. These two sequences can be naturally extended to positive continuous functions $\lambda, \mu:[2,\infty)\to \mathbb R$, see the formulas \eqref{lambda} and \eqref{mu} below.

In this paper we focus on the class of $\Lambda$-coalescents without a dust component.  To put it briefly, in this case the rate at which, within the $n$-coalescents a single lineage merges with some others, diverges as the sample size $n$  tends to infinity. They are characterized by the condition (Pitman \cite{Pi})
\[ \int_{[0,1]} \frac{\Lambda(dp)} p = \infty \ . \]
In particular, they cover $\Lambda$-coalescents coming down from infinity. These are the coalescents with the property that the absorption times $\tilde \tau_n$ are bounded in probability uniformly in $n$. They are distinguished by the criterion (Schweinsberg \cite{Schw})
\[ \int_2^\infty \frac {dx}{\mu(x)}  < \infty \ . \]

We shall analyse the lengths of the whole $n$-coalescents as well as of different parts. They play an important role in the infinite sites model  introduced by Kimura \cite{Kim}. In this model each mutation effects a different site on the DNA. The mutations are distributed along the branches of the coalescent depending on their appearance in the past. Mathematically they build a homogeneous Poisson point process with rate $\theta>0$. Their total number $S(n)$ counts the segregating sites in the sample of size $n$ underlying the coalescent and are thus closely tied to its total length $\ell_n$. Mutations which are located on an external branch appear only once in the sample, these are the singleton polymorphisms (see Wakeley \cite{Wa}, page 103). 
Accordingly their number is linked to the total length $\widehat \ell_n$ of all external branches. Likewise the number of mutations visible repeatedly in the sample corresponds to the total internal length $\widecheck \ell_n$.  

\newpage

\paragraph{Notation.} For two sequences $A_n$ and $B_n$ of positive random variables we write $A_n \stackrel 1\sim B_n$ and $A_n \stackrel P\sim B_n$ if the sequence $A_n/B_n$ converges to 1 in the $L_1$-sense or in probability, respectively. The notation $A_n=O_P(B_n)$ means that the sequence $A_n/B_n$ is tight and $A_n=o_P(B_n)$ that $A_n/B_n$ converges to 0 in probability.

\mbox{}\\
Now we come to the main results of this paper. As a first notion let
\[ \ell_n := \int_0^{\tilde \tau_n} N_n(t)\, dt \ , \ n\ge 1 \ ,\]
be the total length of the coalescent tree.

\begin{Theorem} \label{Thm1}  Assume that the $\Lambda$-coalescent has no dust component. Then as $n \to \infty$
\begin{align} \ell_n \stackrel 1\sim  \int_2^n \frac x{\mu(x)} \, dx \ .  
\label{Theo1}
\end{align}
\end{Theorem}

\noindent
This is an intuitive approximation: if the counting process $N_n$ takes the value $x$, then there are currently $x$ lines and  $1/\mu(x)$, the reciprocal of the rate of decrease,  indicates how long on the average $N_n$  will stay close to  $x$.
Observe that the right-hand integral diverges for $n \to \infty$:  because of Lemma \ref{Lmm1} (i) below the function $\mu(x)/(x(x-1))$ is decreasing, thus $\int_2^n \frac x{\mu(x)} \, dx \ge \frac {2}{\mu(2)}\log (n-1)$. This lower bound is attained by the Kingman coalescent.

Berestycki et al \cite{Be2} conjectured  Theorem \ref{Thm1} and proved it for $\Lambda$-coalescents coming down from infinity in the framework of convergence in probability by using instead of $\mu(x)$ the asymptotically equivalent quantity
\[ \psi(x):= \int_{[0,1]} (e^{-xp}-1+xp) \frac{\Lambda(dp)}{p^2} \ .\]
For the Kingman coalescent the result was earlier obtained by Watterson  \cite{Wat} and for the Bolthausen-Sznitman coalescent by Drmota et al \cite{Dr}.

In the context of the infinite sites model Theorem \ref{Thm1} may be restated directly in terms of the number $S(n)$ of segregating sites as
\[ S(n) \stackrel 1\sim \theta \int_2^n \frac x{\mu(x)} \, dx \ ,  \]
due to the assumption that mutations appear according to a homogeneous Poisson point process with rate $\theta$.

\mbox{}\\ 
Next we define the total external length
\[ \widehat \ell_n := \int_0^{\tilde \tau_n} \widehat N_n(t)\, dt\ , \]
where $\widehat N_n(t)$ denotes the number of external branches extant at time $t$.  In contrast to the previous theorem the following result 
contains a statement only of convergence in probability (but see Theorem \ref{Thm4} below showing that, indeed, $L_1$-convergence holds under the stronger conditions of that theorem).

\begin{Theorem} \label{Thm2}  Suppose that the $\Lambda$-coalescent has no dust component. Then
as $n \to \infty$ 
\begin{align}  \widehat \ell_n \, \stackrel P\sim \, \frac{n^2}{\mu(n)}  \ .
\label{Theo2}
\end{align}
\end{Theorem}

\noindent
An intuitive explanation goes as follows: $\mu_n/n$ is the rate of decrease per individual at time 0. Therefore it is plausible that each external branch length is of order $n/\mu(n)$, and the external branch length in total results approximately in $n^2/\mu(n)$.
The sequence $\mu(n)/n^2$, $n \ge 1$, has the limit $\Lambda(\{0\})/2$, see formula \eqref{mu3} below. Thus the total external lengths diverge in probability if and only if $\Lambda(\{0\})=0$. We point out that $x/\mu(x)$ is a decreasing function (see Lemma \ref{Lmm1} (i) below), therefore the integral appearing in \eqref{Theo1} exceeds the corresponding term in \eqref{Theo2}. 

 The proof rests on a close relation between the functions $\lambda$ and $\mu$ which seems not to be observed until now. To describe it let us introduce another function $\kappa:[2,\infty)\to \mathbb R$ given by
\[  \kappa(x):= \frac{\mu(x)}{x} \ ,\]
which could be named the rate of decrease per capita. Then we have the approximation
\[ \lambda(x) \sim x^2 \kappa'(x)  \]
as $x \to \infty$ (see Lemma \ref{Lmm1} (ii) below).

In the special case of Beta-coalescents Theorem \ref{Thm2} follows from \cite{Da} and \cite{DhMoe}. For $\Lambda$-coalescents with a dust component the picture is rather different. Then $\ell_n/n$ as well as $\widehat \ell_n/n$ converge in distribution to one and the same non-degenerate limit law, see \cite{Moe}. Fluctuation results on the total length or the total external length of $\Lambda$-coalescents with no dust component are known only in special cases \cite{Da,Dr,Ja,Ke,Ta}.

Theorem \ref{Thm2} again allows a reformulation in terms of the infinite site  model. Letting $M_1(n)$ be the number of singletons, we obtain for $\Lambda(\{0\})=0$
\[ M_1(n) \stackrel P \sim \theta \frac{n^2}{\mu(n)} \ , \] 
whereas for $\Lambda(\{0\})>0$ it follows that $M_1(n)$ is asymptotically Poisson with parameter $2\theta/\Lambda(\{0\})$.

\paragraph{Remark.} From our proofs we will gain further insight in the structure of the coalescents. Let for $0<c<1$
\[ \tilde \rho_n:=\inf \{t \ge 0: N_n(t) \le cn \} \]
be the first moment, when the number of lineages falls below $cn$, and let
\[ \ell_n^* := \int_0^{\tilde \rho_n} N_n(t)\, dt \ , \  \widehat \ell_n^* := \int_0^{\tilde \rho_n} \widehat N_n(t)\, dt\]
be the respective lengths up to this moment. Then from \eqref{totLaenge} and \eqref{remark} below we have
\[  \ell_n^* \stackrel 1\sim \int_{cn}^n \frac x{\mu(x)}\, dx \ \text{ and }\ \widehat \ell_n^* \stackrel P\sim (1-c)\frac{n^2}{\mu(n)} \ . \]
The picture gives an illustration. 
\begin{center}
\includegraphics[height=3cm]{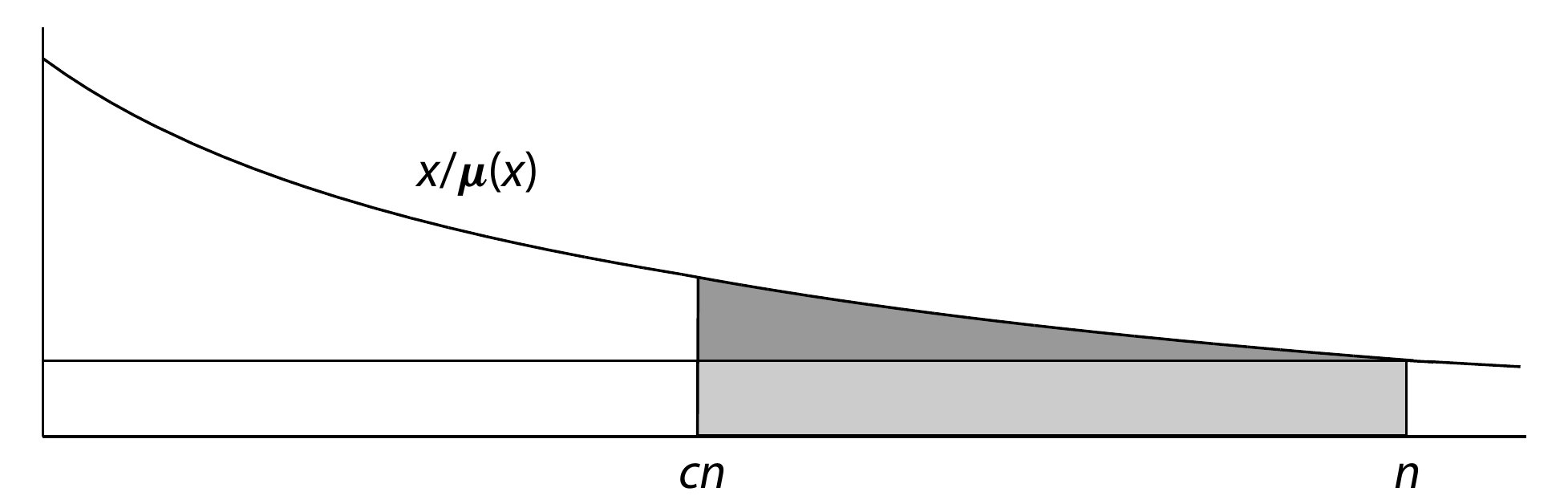}
\end{center}
$\ell_n^* $ and $\widehat \ell_n^*$ are the areas of the total grey region and of the lighter part, respectively.\qed

\mbox{}\\
It is tempting to expect an analogous result for the total internal length
\[ \widecheck \ell_n := \ell_n - \widehat \ell_n \ . \]
This is certainly true in cases where the total external length $\widehat \ell_n$ does not exceed the  total internal length $\widecheck \ell_n$, then we have
\[ \widecheck \ell_n \stackrel P\sim \int_{2}^n  \frac x{\mu(x)}    \, dx- \frac{n^2}{\mu(n)} \sim    \int_{2}^n \Big( \frac x{\mu(x)} - \frac n{\mu(n)}\Big) \, dx \ . \]
To provide some examples let us introduce the following class of coalescents.

\paragraph{Definition.} {\em We call the $\Lambda$-coalescent regularly varying with exponent $0\le \alpha \le 2$ if $\Lambda(\{0\})=0$ and if as $y \to 0$
\[ \int_{(y,1]} \frac{\Lambda (dp)}{p^2} \sim y^{-\alpha} L\Big(\frac 1y\Big) \]
with  a function $L$ slowly varying  at infinity.}

\mbox{}\\
This generalises the notion of Berestycki et al \cite{Be2} of a strong regularly varying coalescent.

\paragraph{Examples.} (i) The Kingman case: If $\Lambda(\{0\}) >0$, then $\mu(x)/x^2 \sim \Lambda(\{0\})/2$, thus
\[ \ell_n \stackrel 1\sim \frac {2 \log n} {\Lambda(\{0\})} \,  \quad \text{ and } \quad \widehat \ell_n \stackrel P\sim \frac 2{\Lambda(\{0\})} \]
as $n \to \infty$. Here the internal total length completely dominates the external ones, and we have
\[ \widecheck \ell_n \stackrel 1\sim \frac {2 \log n}{\Lambda(\{0\})}  \ . \]

(ii) Regularly varying coalescents with exponent $1< \alpha < 2$ come down from infinity as can be easily checked by the above criterion. They
fulfil (see Lemma \ref{Lmm2} (ii) below)
\[ \mu(x) \sim \frac {\Gamma(2-\alpha)}{\alpha-1} x^\alpha L(x)  \]
 as $x \to \infty$. Hence (having in mind how to integrate regularly varying functions, see Thm. 1 (b), Section VIII.9 in Feller \cite{Fe})
\[ \ell_n \stackrel 1\sim \frac{\alpha-1}{(2-\alpha)\Gamma(2-\alpha)}  \frac{n^{2-\alpha}}{L(n)} \ \text{ and }\ \widehat \ell_n \stackrel P\sim \frac{\alpha-1}{\Gamma(2-\alpha)}\frac{n^{2-\alpha}}{L(n)}\]
and consequently
\[\widecheck\ell_n \stackrel P\sim \frac{(\alpha-1)^2}{(2-\alpha)\Gamma(2-\alpha)}  \frac{n^{2-\alpha}}{L(n)}\stackrel P\sim \frac {\alpha-1}{2-\alpha} \widehat \ell_n \ .\]
Here the external and the internal length are of the same order. As an application, one may use the quantity $2- \hat \ell_n/\ell_n$ as an estimator for the parameter $\alpha$. In the special case of Beta($2-\alpha,\alpha$)-coalescents this was already discussed in more detail in Example 9 of \cite{KeSchwWa}. \qed

\mbox{}\\
In some other cases $\widecheck \ell_n$ is of smaller order than $\widehat \ell_n$. Then a large part of internal length will be located close to the coalescent tree's root, where  extremely large  mergers may take over (which is definitely  the case for coalescents not coming down from infinity). Still the internal length may obey the  law of large numbers suggested above.  The following theorem presents a situation of special interest. Define for a slowly varying function $L$ and for $x \ge 1$
\[ L^*(x):= \int_1^x \frac {L(y)}y\, dy \ .\]

\begin{Theorem} \label{Thm3} Assume that the $\Lambda$-coalescent has no dust component and is regularly varying with exponent $\alpha=1$. 
Then
\[\widecheck \ell_n\, \stackrel 1\sim  \int_{2}^n \Big( \frac x{\mu(x)}- \frac n{\mu(n)} \Big) \, dx\]
and
\[   \int_{2}^n \Big( \frac x{\mu(x)}- \frac n{\mu(n)} \Big) \, dx\sim  \frac{nL(n)}{L^*(n)^2}    \]
 as $n \to \infty$.
 \end{Theorem}
 
\noindent
Below we prove (Lemma \ref{Lmm2} (ii)) that the function  $L^*$ is slowly varying at infinity, too, and that $L(x)=o(L^*(x))$ as $x\to \infty$. In comparison with Theorem \ref{Thm2} we see that $\widecheck \ell_n= o_P(\widehat \ell_n)$ for regularly varying coalescents with exponent 1. 
In particular, for the Bolthausen-Sznitman coalescent  Theorems \ref{Thm2} and \ref{Thm3} yield
\[ \widehat \ell_n \stackrel P\sim \frac n{\log n} \quad \text{and} \quad \widecheck \ell_n \stackrel 1\sim \frac n{\log^2 n} \ . \]
These approximations were already obtained in Dhersin and M\"ohle \cite{DhMoe} and Kersting, Pardo and Siri-J\'egousse  \cite{KePa}, respectively.

\mbox{}\\
Our last object of investigation concerns lengths of higher order. A  branch of order $a \ge 2$ is by definition an internal branch carrying a subtree with $a$ leaves out of the original sample. In this context we consider external branches as branches of order 1. Denote  the number of  branches of order $a\ge 1$ present at time $t\ge 0$ as $\widehat  N_{n,a}(t)$, notably $\widehat N_{n,1}(t)=\widehat N_n(t)$.  Then the total length of all these branches   is given by
\[ \widehat \ell_{n,a} = \int_0^{\tilde \tau_n} \widehat N_{n,a}(t)\, dt \ , \ a \ge 1\ . \]

\newpage

\begin{Theorem} \label{Thm4} Suppose that the $\Lambda$-coalescent has no dust component and is regularly varying with exponent $\alpha=1$. Then
\[ \widehat \ell_{n,1} \stackrel 1\sim \frac n{L^*(n)}\ , \]
whereas for $a \ge 2$
\[ \widehat \ell_{n,a} \stackrel 1\sim \frac 1{(a-1)a}\frac {nL(n)}{L^*(n)^2}  \]
as $n \to \infty$.
\end{Theorem}

\noindent
These formulas have interesting applications to the site frequency spectrum. It consists of the counts $M_a(n)$, $a\ge 1$, specifying the   numbers of mutations located on branches of order $a$, which can be distinguished in the various DNAs of the sample. The theorem yields
\[ M_1(n) \stackrel 1\sim \theta \frac n{L^*(n)} \quad \text{and}\quad M_a(n) \stackrel 1\sim \frac{\theta}{(a-1)a} \frac {nL(n)}{L^*(n)^2} \ , \ a\ge 2 \ . \]
A corresponding result in the so-called infinite allele model  was obtained by Basdevant and Goldschmidt \cite{Ba}. They deal   with the allele frequency spectrum and  consider the special case of the Bolthausen-Sznitman coalescent. Analogue results for Beta-coalescents coming down from infinity  were presented by Berestycki, Berestycki and Schweinsberg \cite{Be3} and more generally for strongly regular varying coalescents with exponent $1<\alpha<2$ by Berestycki, Berestycki and Limic \cite{Be2}. We conjecture that these results can be further extended to regular varying coalescents. 
Computational procedures for the general site frequency spectrum are due to Spence et al \cite{Sp}.

Theorem \ref{Thm4} illustrates that mutations show mainly up on external branches, whereas the tree structure of the coalescent becomes visible only at branches of higher order. This reflects that for regularly varying coalescents with exponent $\alpha=1$ mergers occur preferentially at a late time and  close to the MRCA.
Our proof will show that with probability asymptotically equal to 1 any mutation seen in exactly $a\ge 2$ individuals stems from an internal branch of order $a$ arose from one single merger. This  is  similar to  findings of Basdevant and Goldschmidt.

\mbox{}\\
Closing this introduction, we briefly discuss our methods of proof, which differ from other approaches in the literature. They rest upon $L_2$-considerations and elementary martingale estimates, and they may well find further applications as  indicated in two examples below. These methods deal with the time-discrete Markov chain $n=X_0>X_1> \ldots > X_{\tau_n}=1$, embedded in the Markov process $N_n$ (or more generally with decreasing Markov chains). Let
\[ \Delta_i := X_{i-1}-X_i \ , \ i\ge 1 \  ,\]
denote its downwards jump size resulting from the $i$-th merger.
We shall present different laws of large numbers for expressions of the form $\sum_{i=0}^{\rho_n-1} f(X_i)$, with some function $f:[2,\infty)\to  \mathbb R$ and with stopping times $\rho_n$ of the form
\[ \rho_n:= \min \{i\ge 0: X_i \le r_n\} \ ,\] 
where $r_n$, $n \ge 1$, is some sequence of positive numbers.

Our approach specifies the following intuition: 
\begin{align} \label{intuition} \sum_{i=0}^{\rho_n-1} f(X_i) \approx \sum_{i=0}^{\rho_n-1} f(X_i)\frac{\Delta_{i+1}}{\nu(X_i)} \approx \int_{r_n}^n f(x)\, \frac{dx}{\nu(x)}  
\end{align}
with
\[ \nu(x):= \frac{\mu(x)}{\lambda(x)}  \ , \ x \ge 2 \ , \]
thus
\[ \nu(b)= \sum_{k=2}^b (k-1) \binom bk \frac{\lambda_{b,k}}{\lambda(b)}=\sum_{k=2}^b (k-1) \mathbf P(\Delta_{i+1}=k-1\mid X_i=b)   \]
or
\[ \nu(b)=  \mathbf E[\Delta_{i+1} \mid X_i=b] \ ,\ b \ge 2 \ . \]
The rationale behind this intuition is that the differences of both sums in \eqref{intuition} are small, because they stem from a martingale, whereas the second sum may be considered as a Riemann approximation of the right-hand integral. This latter approximation requires that the jump sizes $\Delta_{i+1}$ are small compared to the values $X_i$ of the Markov chain and is only ensured if the  time
\[ \tilde \rho_n := \inf \{ t\ge 0: N_n(t)=X_{\rho_n} \}= \inf \{ t \ge 0: N_n(t)\le r_n \} \]
of entrance into the interval $[1,r_n]$  converges to 0 in probability. (Observe that these random times generalise the notion used in the above remark.) Thus we strive towards small time approximations.

The quadratic variation of the just mentioned martingale will be estimated along the following lines: Since  $\Delta_{i+1}=o_P(X_i)$ in the range of the small time approximation, we have
\begin{align*} \sum_{i=0}^{\rho_n-1} f(X_i)^2 \frac {\Delta_{i+1}^2}{\nu(X_i)^2} = o_P\Big(\sum_{i=0}^{\rho_n-1} f(X_i)^2 \frac{X_i\Delta_{i+1}}{\nu(X_i)^2}\Big) = o_P\Big( \max_{r_n \le x \le n} \frac{xf(x)}{\nu(x)} \sum_{i=0}^{\rho_n-1} f(X_i) \frac{\Delta_{i+1}}{\nu(X_i)} \Big)\ .
\end{align*}
Under suitable conditions not only the right-hand sum but also the maximum is of order $\int_{r_n}^n f(x) \frac{dx}{\nu(x)}$ resulting in
\[ \sum_{i=0}^{\rho_n-1} f(X_i)^2 \frac {\Delta_{i+1}^2}{\nu(X_i)^2}  = o_p\left( \Big(\int_{r_n}^n f(x) \frac{dx}{\nu(x)}\Big)^2\right) \ . \]
According to this pattern we may control the martingale's quadratic variation and its fluctuations.

The paper is organized as follows: Section 2 deals in detail with the rate functions. Section~3 contains our general laws of large numbers. Finally, our theorems are proved in Sections 4 to 7.

\section{Properties of the rate functions}

We extend the above sequences $\lambda(b)$ and $\mu(b)$, $b\ge 2$, of rates    to positive continuous functions $\lambda,\mu:[2,\infty)\to \mathbb R$  via the definitions
\begin{align} \lambda(x):&= \int_{[0,1]}\big(1-(1-p)^x-xp(1-p)^{x-1}\big)\, \frac{\Lambda(dp)}{p^2} \ , \label{lambda}\\
\mu(x) :&= \int_{[0,1]} (xp-1+(1-p)^x) \, \frac{\Lambda(dp)}{p^2} \ ,
\label{mu}
\end{align}
with $x \ge 2$. Note that $\mu(x)\ge \lambda(x)$ for all $x \ge 2$, since the integrands fulfil the corresponding inequality. Recall our notion
\[ \kappa(x) :=  \frac {\mu(x)}x  \ , \ x\ge 2 \ . \]

\begin{Lemma}\label{Lmm1} {\em (i)}  The functions $\lambda(x)$ and $\kappa(x)$, $x \ge  2$, are increasing in $x$, and the functions $\lambda(x)/(x(x-1))$ and $\kappa(x)/(x-1)$, $x \ge 2$, are decreasing.

{\em (ii)}   For any $0<\chi < 1$ as $x \to \infty$ 
\begin{align*} \lambda(x) = x^2 \kappa'(x)(1+ O(x^{-\chi})) \ .\end{align*}

{\em (iii)}   The  $\Lambda$-coalescent has no dust component if and only if
$\kappa(x)\to \infty $
 as $x \to \infty$ and then
$\kappa(x) = o(\lambda(x))$.
\end{Lemma}

\begin{proof}
(i)  The function $(1-p)^x+ xp(1-p)^{x-1}= (1-p)^{x-1}(1+(x-1)p)$, $x \ge 2$, respectively its logarithm, has a negative derivative. From \eqref{lambda} we  thus obtain  monotonicity of $\lambda(x)$.
Moreover, a partial integration yields
\begin{align} \frac{\lambda(x)}{x(x-1)}= \frac{\Lambda(\{0\})}2+ \int_0^1  y(1-y)^{x-2}\int_{(y,1]}\frac {\Lambda(dp)}{p^2}\, dy 
\label{lambda2}
\end{align}
implying that $\lambda(x)/(x(x-1))$ is decreasing for $x \ge 2$.

Similarly, from \eqref{mu} and a partial integration we have
\begin{align}\label{mu2} \kappa(x)= \frac {\mu(x)}x = \frac{x-1}2\, \Lambda(\{0\}) +\int_0^1 (1- (1-y)^{x-1} ) \int_{(y,1]}\frac {\Lambda(dp)}{p^2}\, dy \ ,\end{align}
which is increasing in $x$, and by another partial integration
\begin{align} \frac{\mu(x)}{x(x-1)} = \frac{\Lambda(\{0\})}2+ \int_0^1 (1-z)^{x-2} \int_z^1\int_{(y,1]} \frac {\Lambda(dp)}{p^2} \, dy\, dz \ ,  
\label{mu3}
\end{align}
a decreasing function in $x$.

(ii) From \eqref{mu2} we have
\begin{align}\kappa'(x)= \frac{\Lambda(\{0\})}2+ \int_0^1 (1-y)^{x-1}\log \frac 1{1-y} \int_{(y,1]}\frac {\Lambda(dp)}{p^2}\, dy\ . 
\label{kappaprime}
\end{align}
From concavity $(1-y)\log \frac1{1-y} \le y$ for $y \in [0,1]$, also $(1-y)\log \frac 1{1-y} \sim y + O(y^2)$ for $y \to 0$. Thus, combining \eqref{lambda2} and \eqref{kappaprime}, for  $0<\chi <1$ and a suitable $c\ge 1$
\begin{align*}
0&\le \frac{\lambda(x)}{x(x-1)} - \kappa'(x) \\&\le c\int_0^{x^{-\chi}} y^2 (1-y)^{x-2} \int_{(y,1]}\frac {\Lambda(dp)}{p^2}\, dy +\int_{x^{-\chi}}^1  y(1-y)^{x-2}\int_{(y,1]}\frac {\Lambda(dp)}{p^2}\, dy\\
&\le c x^{-\chi} \int_0^{x^{-\chi}} y (1-y)^{x-2} \int_y^1\frac {\Lambda(dp)}{p^2}\, dy + (1-x^{-\chi})^{x-2} \int_{x^{-\chi}}^1  y\int_{(y,1]}\frac {\Lambda(dp)}{p^2}\, dy\\
&\le c x^{-\chi} \int_0^{1} y (1-y)^{x-2} \int_y^1\frac {\Lambda(dp)}{p^2}\, dy + e^{- x^{-\chi}(x-2)}\int_{0}^1  y\int_{(y,1]}\frac {\Lambda(dp)}{p^2}\, dy\\
&\le c x^{-\chi} \frac{\lambda(x)}{x(x-1)} + e^{2-x^{1-\chi}} \Lambda([0,1]) \ .
\end{align*} 
Since $\lambda$ is an increasing function, this implies
\[ \kappa'(x) = \frac{\lambda(x)}{x^2} + O\Big( \frac{\lambda(x)}{x^{2+\chi}} \Big)\ ,\]
which is equivalent to our claim.

(iii) From \eqref{mu} and the definition of $\kappa$ it follows by monotone convergence that
\begin{align*}
\kappa(x) \to \int_{[0,1]} \frac{\Lambda(dp)}p \ .
\end{align*}
This formula implies our first claim. As to the second one 
we have for  $2 \le x_0 \le x$
\[ \kappa(x) = \kappa(x_0) + \int_{x_0}^x \kappa'(y)\, dy \ . \]
From part (ii) and since $\lambda$ is increasing, if $x_0$ is sufficiently large
\[ \kappa(x) \le \kappa(x_0) +2 \int_{x_0}^x \frac{\lambda(y)}{y^2}\, dy \le \kappa(x_0) + 2  \lambda(x)\int_{x_0}^x \frac{dy}{y^2} \le \kappa(x_0) + 2 \frac{ \lambda(x)}{x_0}\ . \]
Thus, from $\kappa(x)\to \infty$ it follows that $\lambda(x)\to \infty$ as $x \to \infty$, and we
obtain for any $x_0 \ge 2$
\[ \limsup_{x \to \infty} \frac{\kappa(x)}{\lambda(x)} \le \frac 2{x_0} \ . \]
This implies our second claim.
\end{proof}

\noindent
First consequences of the previous results are contained in the next lemma.

\begin{Lemma}  \label{Lmm2} Let the $\Lambda$-coalescent be regularly varying with exponent
$\alpha$. 

{\em (i)} If $0 \le \alpha < 2$, then
\[ \lambda(x) \sim \Gamma(2-\alpha) x^\alpha L(x)  \]
as $x \to \infty$, and for $k \ge 2$
\[ \binom bk \frac {\lambda_{b,k}}{\lambda(b)} \to \frac{\alpha}{\Gamma(2-\alpha)} \frac{\Gamma(k-\alpha)}{k!} \]
as $b \to \infty$.

{\em (ii)} The $\Lambda$-coalescent has no dust component if and only if $\int_1^\infty x^{\alpha -2} L(x) \, dx=\infty$. Then $\alpha\ge 1$ and we have
\[ \kappa(x) \sim\begin{cases} \frac {\Gamma(2-\alpha)}{\alpha-1} x^{\alpha-1} L(x) & \text{for } 1<\alpha <2 \\  L^*(x) &\text{for }\alpha =1 \end{cases}  \]
as $x \to \infty$, with a  function $L^*$ given by
\[ L^*(x):= \int_1^x \frac{ L(y)}y \, dy\ , \ x \ge 1 \ . \]
$L^*$ is slowly varying at infinity and satisfies   $L(x) = o(L^*(x))$ as $x \to \infty$.
\end{Lemma}

\noindent
For $1<\alpha < 2$ the convergence result on $\lambda_{b,k}$ has already been obtained by Delmas et al \cite{De}. In this case they  also have asymptotic estimates on $\lambda(x)$ and $\mu(x)$ under the assumption that the slowly varying function $L$ is constant.

\begin{proof} (i)
Let $1\le k \le b$ be natural numbers. Starting from the identity
\[ 1- (1-p)^b- bp(1-p)^{b-1} - \cdots - \binom bk p^k (1-p)^{b-k} = \binom bk (b-k) \int_0^p y^k (1-y)^{b-k-1} \, dy\ , \]
we obtain
\begin{align*} \lambda(b)- \sum_{j=2}^k \binom bj \lambda_{b,j} &= \binom bk (b-k)\int_{[0,1]}  \int_0^p y^k (1-y)^{b-k-1} \, dy\, \frac{\Lambda(dp)}{p^2}\\
&= \binom bk (b-k)\int_0^1   y^k (1-y)^{b-k-1}   \int_{(y,1]}  \frac{\Lambda(dp)}{p^2} \, dy \ .
\end{align*}
Taking account of the definition of regular varying coalescents, it is no loss of generality to specify the function $L$ in such a way  that $\int_{(y,1]}  p^{-2}\Lambda(dp) =y^{-\alpha}L(1/y)$ for $0<y\le1$. It follows that
\begin{align*}
\lambda(b)- \sum_{j=2}^k \binom bj \lambda_{b,j} &= \binom bk (b-k)\int_0^1 y^{k-\alpha} (1-y)^{b-k-1} L\Big( \frac 1y\Big) \, dy \\
&= \binom bk (b-k)b^{\alpha-k-1}L(b) \int_0^b z^{k-\alpha} \Big(1-\frac zb\Big)^{b-k-1} \frac{L(b/z)}{L(b)} \, dz \ ,
\end{align*}
where for $k=1$ we set  the value of the left-hand sum equal to $0$. 

Since the function $L$ is slowly varying, the right-hand integrand converges pointwise to the limit 
$z^{k-\alpha} e^{-z}$. Also by the fundamental representation theorem of slowly varying functions (see Feller \cite{Fe} Section VIII.9, Corollary) we have $L(b) \sim c\exp( \int_0^b \eta(z)z^{-1} \, dz)$ with a constant $c>0$ and a function $\eta(z)=o(z)$ as $z \to \infty$. This implies  that for any $\varepsilon>0$ we have $L(b/z) \le z^{\varepsilon} L(b)$ for  $z\ge 1$ and $L(b/z) \le z^{-\varepsilon} L(b)$ for  $z\le 1$ if only $b$ is sufficiently large. Hence, for large $b$ we may dominate the above integrand by the function $z^{k-\alpha} \max(z^{\varepsilon},z^{-\varepsilon})e^{-z/2}$. By the assumption  $\alpha<2$ it is integrable for $k \ge 1$ if we choose $\varepsilon $  small enough. Thus by dominated convergence
\[ \lambda(b)- \sum_{j=2}^k \binom bj \lambda_{b,j} \sim \frac{b^\alpha L(b)}{k!} \int_0^\infty z^{k-\alpha}e^{-z}\, dz=\frac{b^\alpha L(b)\Gamma(k-\alpha+1)}{k!} \]
as $b \to \infty$, or in other terms
\[\lambda(b) \sim \Gamma(2-\alpha) b^\alpha L(b)   \ \text{ and } \ \binom bk \lambda_{b,k} \sim \frac{\alpha\Gamma(k-\alpha)}{k!} b^\alpha L(b) \]
for $k \ge 2$. This implies our claim.

(ii) From part (i) and Lemma \ref{Lmm1} (ii) we obtain
\[ \kappa'(x) \sim  \Gamma(2-\alpha) x^{\alpha-2} L(x)\]
as $x \to \infty$.
In view of Lemma \ref{Lmm1} (iii) we are  in the dustless case if and only if $\kappa(x)\to \infty$, that is if and only if the integral $\int_1^\infty x^{\alpha-2} L(x)\, dx$ is divergent. This  implies $\alpha \ge 1$ and the  claimed asymptotic formulas for $\kappa$.

Finally, since $L$ is slowly varying at infinity, we have for $c>1$   
\begin{align} L^*(cx)-L^*(x) = \int_{x}^{cx} \frac {L(y)}y \, dy \sim L(x) \int_{x}^{cx} \frac {1}y \, dy = L(x) \log c 
\label{Lstern}
\end{align}
as $x \to \infty$. Because $L^*$ is increasing, this implies for any $d>1$ and large $x$
\[ 0 \le \frac{L^*(cx)}{L^*(x)}-1 \le \frac{L^*(cx)-L^*(x)}{L^*(x)-L^*(x/d)}\sim \frac{\log c}{\log d}\]
as $x\to \infty$. Since $d$ may be chosen arbitrarily large, it follows that $L^*$ is slowly varying. Consequently, choosing $c=e$ in \eqref{Lstern},
\[ \frac {L(x)}{L^*(x)} \sim \frac{ L^*(ex)-L^*(x)}{L^*(x)} =o(1)\ . \]
This finishes our proof.
\end{proof}

\paragraph{Examples.} We consider regularly varying coalescents with exponent $\alpha=1$.

(i) Let $L(x)= (\log x)^\chi $  with exponent $\chi >-1$. Then
\[ L^*(x)= \int_0^{\log x}  y^\chi \, dy = \frac 1{\chi+1}  (\log x)^{\chi+1} = \frac 1{\chi+1} L(x)\log x \ . \]
For $-1 < \chi \le 0$ these coalescents neither have a dust component nor come down from infinity, and for $\chi >0$ they come down from infinity. The case $\chi=0$ covers the Bolthausen-Sznitman coalescent.

(ii) Let $L(x)= e^{(\log x)^\chi}$ with $0<\chi < 1$. Then
\[L^*(x)= \int_0^{\log x} e^{y^\chi} \, dy \sim \frac 1{\chi}  e^{(\log x)^\chi} (\log x)^{1-\chi}=\frac 1{\chi}L(x)  (\log x)^{1-\chi} \ . \]
These coalescents come down from infinity. \qed

\section{Some general laws of large numbers}

The laws of large numbers of this section  apply not only to branch lengths of $\Lambda$-coalescents. As explained  in the introduction, they concern the  discrete-time Markov chains $X=(X_i)_{i\in \mathbb N_0}$ embedded  in the lineage counting processes. For notational ease we do  not account here for the dependence of the chains $X$ on $n$. Thus    $n=X_0 > X_1 > \cdots > X_{\tau_n-1} > X_{\tau_n}=1$ denote the states which the process $N_n$ is successively visiting, and  $\tau_n:=\min \{ i \ge 0: X_i=1\}$ is the respective number of merging events. For convenience we set $X_i=1$ for all natural numbers $i >\tau_n$. Also let for $i\ge 1$ be $W_i$ the waiting times of the process $N_n$ in the states $X_i$ and
\[ \Delta_i := X_{i-1}-X_{i} \ . \]
Recall that
\[ \mathbf E[ \Delta_{i+1} \mid X_i=b] = \nu(b):= \frac{\mu(b)}{\lambda(b)}\ . \]
For sequences $r_n, s_n\ge 2$, $n \ge 1$, of positive numbers we set
\[ \rho_n:= \min\{ i\ge 0: X_i\le r_n\} \ , \ \sigma_n:= \min\{ i\ge 0: X_i\le s_n\}\ , \ \tilde \rho_n := \inf\{t \ge 0: N_n(t)\le r_n \} \ .\]

\begin{Proposition}  \label{Prp1} Assume that the coalescent has no dust component.
Let $f:[2,\infty) \to \mathbb R$ be a  non-negative
function with the property that $x^\beta f(x)$ is  increasing and $x^{-\beta}f(x)$ is decreasing in $x$ for some $\beta >0$. Let $2\le r_n\le  s_n\le n$, $n \ge 1$, be two sequences of numbers fulfilling 
\begin{align}  
r_n \le \gamma s_n 
\label{rnrn}
\end{align}
for all $n \ge 1$ and some $\gamma <1$.  Also assume that
\begin{align} \tilde \rho_n \stackrel P \to 0 
\label{sigma0}
\end{align}
as $n \to \infty$. Then,  as $n \to \infty$
\[ \sum_{i=\sigma_n}^{\rho_n-1} f(X_i) \ \stackrel 1\sim \int_{r_n}^{s_n} f(x) \, \frac {dx}{\nu(x)} \ . \]
Also, as $n \to \infty$
\[ \mathbf E\Big[ \sum_{i=\rho_n}^{\tau_n-1} f(X_i) \Big]= O\Big( \int_2^{r_n} f(x) \, \frac {dx}{\nu(x)}\Big)  \ . \]
\end{Proposition}

\noindent
Due to the assumption \eqref{sigma0}  this result addresses the coalescent's evolution in only  a short initial  time interval. This takes double effect: first, as seen from the next lemma,  the chain is kept away from the occurrence of huge jumps $\Delta_i$  being of the same order as the chain's values. Second, the chain $X$ is prevented  from taking too small values where larger fluctuations may become obstructive.  Note that for any $\Lambda$-coalescent the passage times $\inf\{ t \ge 0: N_n(t) \le r\}$ below some number $r>1$ are bounded away from 0 uniformly in $n\in \mathbb N$. Therefore the assumption \eqref{sigma0}  enforces that 
\begin{align} r_n \to \infty
\label{rninfty}
\end{align}
as $n \to \infty$. 
Actually,  both requirements are equivalent if the coalescent comes down from infinity, otherwise the assumption \eqref{sigma0}  is the more incisive one.

\paragraph{Example: total number of mergers.} If we choose $f\equiv 1$ and $s_n=n$, then $\sum_{i=0}^{\rho_n-1} f(X_i)$ is equal to the number of mergers up to time $\tilde \rho_n$. For coalescents coming down from infinity we may consider any divergent sequence $r_n\le \gamma n$. In particular, since $\int_2^\infty\frac {dx}{\nu(x)}=\infty$,  we may choose the $r_n$ in such a way that $r_n=o\big(\int_{r_n}^n \frac {dx}{\nu(x)}\big)$. This implies that the number of  mergers after the moment $\tilde \rho_n$ are of negligible order and that for  the total number $\tau_n$ of mergers we have
\[ \tau_n \stackrel 1 \sim \int_2^n \frac {dx}{\nu(x)} \ .\] 
Away from coalescents coming down from infinity the scope of this formula is unclear. For the Bolthausen-Sznitman coalescent it is valid (see \cite{GoMa}). \qed

\mbox{}\\
The proof of Proposition \ref{Prp1} is prepared by the next lemma.

\begin{Lemma} \label{Lmm3}  {\em (i)} If the coalescent has no dust component and if $\int_{r_n}^n \mu(x)^{-1}\, dx \to 0$, then as $n \to \infty$ \[\mathbf E[\tilde \rho_n]  \to 0\ . \]

 {\em (ii)} If $\tilde \rho_n \stackrel P\to 0$, then for any $\eta >0$ as $n \to \infty$
\[ \mathbf P( \Delta_{i+1} > \eta X_i \text{ for some } i< \rho_n) \to 0\ .\]
\end{Lemma}

\begin{proof} (i) Given $X$ the waiting times $W_i$ are exponential with expectation $1/\lambda(X_i)$. Therefore
\begin{align*}
\mathbf E[\tilde \rho_n] = \mathbf E\Big[ \sum_{i=0}^{\rho_n-1} W_i \Big] = \sum_{i=0}^{n-1}\mathbf E\Big[  \frac 1{\lambda(X_i)}\ ; \ X_i > r_n \Big] \ .
\end{align*}
Also $\mathbf E [\Delta_{i+1} \mid X_i] = \nu(X_i)$ a.s., thus by the Markov property
\[ \mathbf E[\tilde \rho_n] = \sum_{i=0}^{n-1} \mathbf E \Big[ \frac {\Delta_{i+1}}{\lambda(X_i)\nu(X_i)}\ ; \ X_i > r_n\Big]= \mathbf E \Big[\sum_{i=0}^{\rho_n-1}\frac{\Delta_{i+1} }{\mu(X_i)} \Big] \ .  \]
From Lemma \ref{Lmm1} (i) we have that $\mu(x)=x\kappa(x)$ is increasing. Also $\Delta_{i+1} \le X_i$, hence
\[ \mathbf E[\tilde \rho_n] \le \mathbf E \Big[ \sum_{i=0}^{\rho_n-2} \int_{X_{i+1}}^{X_{i}} \frac{dx}{\mu(x)} + \frac{X_{\rho_n-1}}{\mu(X_{\rho_n-1})} \Big]= \mathbf E\Big[ \int_{X_{\rho_n-1}}^{n} \frac{dx}{\mu(x)}  + \frac{1}{\kappa(X_{\rho_n-1})} \Big]\ ,\]
and since $\kappa$ is an increasing function, we end up with the estimate
\[ \mathbf E[\tilde \rho_n] \le \int_{ r_n}^{n} \frac{dx}{\mu(x)} + \frac 1{\kappa(r_n)} \ . \]
Our assumptions imply $r_n \to \infty$. Therefore, since there is no dust component, we have $\kappa(r_n)\to \infty$ by Lemma \ref{Lmm1} (iii). This entails our claim.

(ii) For $b \ge 2/\eta$ we have
\begin{align*}
\mathbf P( \Delta _{i+1} > \eta X_i \mid X_i=b) &= \sum_{k >\eta b}\frac 1{\lambda (b)}  \int_{[0,1]} \binom bk p^k(1-p)^{b-k} \frac{\Lambda (dp)}{p^2} \\
&\le \frac 1{\lambda (b)}  \int_{[0,1]}  \sum_{k >\eta b} \frac {2}{\eta^2} \binom {b-2}{k-2} p^{k-2}(1-p)^{(b-2)-(k-2)}\, \Lambda(dp)\\
&\le \frac {c}{ \lambda(b)} \ ,
\end{align*}
with $c:=2\Lambda([0,1])/\eta^2$. 

Let  $\chi_0:=0$ and $\chi_{i}:= W_0+ \cdots + W_{i-1}$, $i\ge 1$, which is the moment of the $i$-th jump.  Then
\begin{align*}
\mathbf P( \Delta_{i+1} > \eta X_i &\text{ for some } i< \rho_n\, , \, \tilde \rho_n \le 1) \\&\le  \mathbf P( \Delta_{i+1} > \eta X_i\, ,\, X_i > r_n\, ,\, \chi_{i} \le 1 \text{ for some } i < n)\\
&\le \sum_{i=0}^{n-1} \mathbf E\big[ \mathbf P(\Delta_{i+1} > \eta X_i \mid X_i) \, ; \, X_i > r_n, \chi_{i} \le 1\big] \\
& \le  \sum_{i=0}^{n-1} \mathbf E\Big[ \frac c{\lambda (X_i)}\, ; \, X_i > r_n, \chi_{i} \le 1\Big]\ .
\end{align*}
Also, because $\lambda(x)$ is increasing, for $X_i \ge 2$
\begin{align*} \mathbf E [ W_iI_{\{W_i \le 1\}} \mid X_i] = \int_0^1t \lambda (X_i)e^{-\lambda(X_i) t}\, dt = \frac 1{\lambda(X_i)} \int_0^{\lambda(X_i)} ue^{-u} \, du \ge \frac d{\lambda(X_i)}  
\end{align*}
with $d:= \int_0^{\lambda(2)} ue^{-u} \, du >0$. This allows for the estimate
\begin{align*}\mathbf P( \Delta_{i+1} > \eta X_i &\text{ for some } i< \rho_n\, , \, \tilde \rho_n \le 1) \\
&\le\frac cd \sum_{i=0}^{n-1}  \mathbf E\Big[  W_i \, ; \,X_i > r_n, W_i \le 1, W_0+ \cdots + W_{i-1}  \le 1\Big]\\
&\le\frac cd  \mathbf E\Big[  \sum_{i=0}^{\rho_n-1} W_i I_{\{ W_0+\cdots+W_i\le 2\}}\Big]\\
&\le\frac cd \mathbf E\Big[ 2\wedge \sum_{i=0}^{\rho_n-1} W_i \Big] \\
&= \frac cd \mathbf E[ 2\wedge \tilde \rho_n] \ ,
\end{align*}
and consequently
\[ \mathbf P( \Delta_{i+1} > \eta X_i \text{ for some } i< \rho_n ) \le \frac cd \mathbf E[ \tilde \rho_n \wedge 2]+ \mathbf P(\tilde \rho_n >1)\ . \]
Thus by assumption and dominated convergence our claim follows.
\end{proof}

\begin{proof}[Proof of Proposition \ref{Prp1}]
(i) We start with some preliminary estimates. Since $x^{-\beta} f(x)$ is a decreasing function, we have for $0<\chi<1$, $0\le y\le  \chi z  $ and $y\le z-2$
\begin{align*}
yf(z) \le z^{\beta}  \int_{z-y}^{ z} x^{-\beta} f(x)\, dx \le  z^\beta (z-y)^{-\beta}\int_{z-y}^{ z} f(x)\, dx \le (1-\chi)^{-\beta}  \int_{z-y}^{ z} f(x)\, dx \ .
\end{align*}
Similarly, using that $x^\beta f(x)$ is increasing,  we obtain a lower bound. Altogether, for $0<\chi<1$, $0\le y\le \chi z$ and $y \le z-2$
\begin{align}
(1-\chi)^{\beta}  \int_{z-y}^{ z} f(x)\, dx \le yf(z) \le (1-\chi)^{-\beta}  \int_{z-y}^{ z} f(x)\, dx \ .
\label{fint1}
\end{align}

Moreover, for any $\varepsilon >0$ there is an $\eta>0$ such that
\begin{align}(1-\varepsilon) \int_{(1-\eta)r_n}^{s_n} f(x) \, dx \le  \int_{r_n}^{s_n} f(x) \, dx\le (1+\varepsilon) \int_{r_n}^{(1-\eta)s_n} f(x) \, dx
\label{fint2}
\end{align}
for all $n$. We prove the left-hand inequality. Let $a$ be the affine function mapping the interval $[(1-\eta)r_n, r_n]$ onto $[(1-\eta)r_n, r_n/\gamma]$. Substituting $y=a(x)$, this implies
\[ dy= \frac{\eta + \gamma^{-1}-1}{\eta} \, dx \ . \]
Moreover, since $\gamma < 1$, $a(x)\ge x$ for $x\ge (1-\eta)r_n$.
Therefore by monotonicity of $x^\beta$ and  $x^\beta f(x)$ we have with $\gamma $ as in \eqref{rnrn}
\begin{align*} \int_{(1-\eta)r_n}^{r_n} f(x) \, dx &\le ((1-\eta)r_n)^{-\beta} \int_{(1-\eta)r_n}^{r_n} f(x)x^\beta \, dx \\
&
\le ((1-\eta)r_n)^{-\beta} \int_{(1-\eta)r_n}^{r_n} f(a(x))a(x)^\beta \, dx
\\
&= ((1-\eta)r_n)^{-\beta}  \frac \eta {\eta+\gamma^{-1}-1}\int_{(1-\eta)r_n}^{r_n/\gamma} f(y)y^\beta \, dy\\
& \le (\gamma(1-\eta))^{-\beta} \frac \eta {\eta+\gamma^{-1}-1}\int_{(1-\eta)r_n}^{r_n/\gamma} f(y) \, dy\ .
\end{align*}
From condition \eqref{rnrn} it follows that
\[ \int_{(1-\eta)r_n}^{r_n} f(x) \, dx  \le (\gamma(1-\eta))^{-\beta} \frac \eta {\eta+\gamma^{-1}-1}\int_{(1-\eta)r_n}^{s_n} f(x) \, dx \le \varepsilon \int_{(1-\eta)r_n}^{s_n} f(x) \, dx \]
for sufficiently small $\eta >0$. This implies the left-hand inequality  of \eqref{fint2}. The other one follows similarly, now using the monotonicity of $x^{-\beta}f(x)$.

We note that $(x-1)/x^2$ is a decreasing function for $x \ge 2$. Therefore, in view of Lemma~1~(i) the functions $\lambda(x)$ and $\mu(x)$ are increasing and $\lambda(x)/x^3$ and $\mu(x)/x^3$ are decreasing for $x \ge 2$. Thus the function $f(x)/\nu(x)=f(x)\lambda(x)/\mu(x)$ fulfils the same assumptions as $f(x)$, with $\beta$ replaced by $\beta + 3$. Accordingly, we shall use the preceding estimates with $f(x) $ replaced by $f(x)/\nu(x)$ (and $\beta $ replaced by $\beta+3$).

(ii) Next we develop  Riemann approximations of certain random sums. For $0<\eta < 1$ and $b \ge 2$ let
\[ \mu_\eta(b) := \sum_{2 \le k \le \eta b} (k-1)\binom bk \int_{[0,1]}  p^k(1-p)^{b-k} \, \frac{\Lambda(dp)}{p^2} \ . \]
In view of the well-known formula for the second factorial moment of binomials,
\begin{align*} \mu(b)-\mu_\eta(b) &= \sum_{\eta b < k \le b} (k-1) \binom bk\int_{[0,1]}   p^k(1-p)^{b-k}\, \frac{\Lambda(dp)}{p^2} \\&
\le \frac 1{\eta b} \int_{[0,1]}  \sum_{k=0}^b k(k-1) \binom bk p^k(1-p)^{b-k}\, \frac{\Lambda(dp)}{p^2}\\
&= \frac 1{\eta b} \int_{[0,1]}  b(b-1)p^2\, \frac{\Lambda(dp)}{p^2} = \frac{(b-1)\Lambda([0,1])}\eta \ .
\end{align*}
From Lemma \ref{Lmm1} (iii) we have $\mu(b)/b\to \infty$ in the dustless case, hence we obtain for any $\eta>0$
\begin{align}\label{mueps} \frac {\mu_\eta(b)}{\mu(b)} \to 1 \end{align}
as $b \to \infty$.

Also let
\[ \nu_\eta(b):= \frac {\mu_\eta(b)}{\lambda (b)} = \mathbf E[ \Delta_{i+1}I_{\{\Delta_{i+1} \le \eta b\}} \mid X_i=b ] \ .\]
Note that $\nu_\eta(b)>0$ for $b\ge 2/\eta$. Therefore,    for  $0<\eta <1$  and natural numbers $n$ satisfying $r_n \ge 2/\eta$ we may define  the random variables
\[ R_n=R_{n,\eta}:= \sum_{i=\sigma_n}^{\rho_n-1} f(X_i) \frac{ \Delta_{i+1}I_{\{\Delta_{i+1} \le \eta X_i\}}}{\nu_{\eta}(X_i)}  \ . \]
Given $\eta$ these random variables are in view  of \eqref{rninfty} well-defined up to finitely many $n$. We shall use them below as an intermediate approximation to $\sum_{i=\sigma_n}^{\rho_n-1} f(X_i)$. We estimate $R_n$ from above and below. From \eqref{fint1} with $z=X_i$,  $y=\Delta_{i+1}$ and $\chi=\eta$  and with $f(x)/\nu(x)$ replacing $f(x)$ we have on the event that $\Delta_{i+1} \le \eta X_i$
\[ \frac{f(X_i)}{\nu(X_i)} \Delta_{i+1} \le (1-\eta)^{-\beta-3} \int_{X_{i+1}}^{X_i} f(x)\, \frac{dx}{\nu(x)} \ , \]
consequently 
\begin{align*}
R_n \le (1-\eta)^{-\beta -3} \sum_{i=\sigma_n}^{\rho_n-1} \frac {\nu(X_i)}{\nu_\eta(X_i)} I_{\{\Delta_{i+1} \le \eta X_i\}}\int_{X_{i+1}}^{X_i} f(x)\, \frac{dx}{\nu(x)} 
\end{align*}
and by definition of $\rho_n$
\begin{align}
R_n \le (1-\eta)^{-\beta-3} \sup_{b \ge r_n} \frac{\mu(b)}{\mu_\eta(b)} \int_{r_n(1-\eta)}^{s_n} f(x) \, \frac{dx}{\nu(x)} \ .
\label{Rn4}
\end{align}

Therefore, in view of   \eqref{fint2} and \eqref{mueps} and since $r_n\to \infty$,  there is for given $\varepsilon >0$ an $\eta>0$ fulfilling
\begin{align} R_n \le (1+\varepsilon) \int_{r_n}^{s_n} f(x) \, \frac{dx}{\nu(x)} 
\label{Rn}
\end{align}
for all $n$ sufficiently large. Similarly, for given $\varepsilon>0$ there is an $\eta>0$ satisfying for large $n$ the inequality
\begin{align}
R_n \ge (1-\varepsilon) \int_{r_n}^{s_n} f(x) \, \frac{dx}{\nu(x)}  \ \text{ on the event } \ \{ \Delta_{i+1} \le \eta X_i \text{ for all } i < \rho_n\} \ .
\label{Rn2}
\end{align}

(iii) Now observe that the random variables $M_0:=0$ and
\[ M_k := \sum_{i=0}^{k\wedge \tau_n-1} \Big(  f(X_i) \frac{ \Delta_{i+1}I_{\{\Delta_{i+1} \le \eta X_i\}}}{\nu_{\eta}(X_i)}-f(X_i)  \Big)\ , \ k \ge 1 \ , \]
 build a martingale $M=(M_k)_{k \ge 0}$. The optional sampling theorem yields
\begin{align*}
\mathbf E\Big[\Big(R_n-\sum_{i=\sigma_n}^{\rho_n-1} f(X_i) \Big)^2\Big] &= \mathbf E[(M_{\rho_n}- M_{\sigma_n})^2]\\
&\le \mathbf E\Big[ \sum_{i=\sigma_n}^{\rho_n-1} \frac{f(X_i)^2 \Delta_{i+1}^2}{\nu_\eta(X_i)^2}I_{\{\Delta_{i+1} \le \eta X_i\}}\Big] \\
&\le \eta \mathbf E\Big[ \sum_{i=\sigma_n}^{\rho_n-1} \frac{f(X_i)^2X_i \Delta_{i+1}}{\nu_\eta(X_i)^2}I_{\{\Delta_{i+1} \le \eta X_i\}}\Big] \ .
\end{align*}
Letting $x_n$ be the point where the function $xf(x)/\nu_\eta(x)$ takes its maximal value within the interval $[r_n,s_n]$, it follows
\begin{align*}
\mathbf E\Big[\Big(R_n-\sum_{i=\sigma_n}^{\rho_n-1} f(X_i) \Big)^2\Big]  \le \eta \frac{x_nf(x_n)}{\nu_\eta(x_n)} \mathbf E\Big[ \sum_{i=\sigma_n}^{\rho_n-1} \frac{f(X_i) \Delta_{i+1}}{\nu_\eta(X_i)}I_{\{\Delta_{i+1} \le \eta X_i\}}\Big]=\eta \frac{x_nf(x_n)}{\nu_\eta(x_n)} \mathbf E\big[R_n\big]\ .
\end{align*}
By the assumption \eqref{rnrn} it follows that there are numbers $\xi_n$, $n \ge 1$, such  that 
\[r_n \le \gamma \xi_n \le x_n \le \xi_n \le s_n\ .\] 
Using monotonicity of $x^{\beta+3}f(x)/\nu(x)$ and \eqref{fint1} with $z=\xi_n$, $y=(1-\gamma)\xi_n $, $\chi=1-\gamma$, it follows that
\[ (1-\gamma) \frac{x_n f(x_n)}{\nu(x_n)}  \le \Big(\frac {\xi_n}{ x_n}\Big)^{\beta+2} (1-\gamma)\frac{\xi_n f(\xi_n)}{\nu(\xi_n)}\le \gamma^{-2\beta-5} \int_{\gamma \xi_n}^{\xi_n} f(x) \, \frac{dx}{\nu(x)}  \le \gamma^{-2\beta-5} \int_{r_n}^{s_n} f(x) \, \frac{dx}{\nu(x)}\ , \]
hence
\[ \mathbf E\Big[\Big(R_n-\sum_{i=\sigma_n}^{\rho_n-1} f(X_i) \Big)^2\Big]  \le \eta \frac{\gamma^{-2\beta-5}}{1-\gamma} \sup_{x \ge 2/\eta }\frac{\nu(x)}{\nu_\eta(x)} \mathbf E[R_n] \int_{r_n}^{s_n} f(x) \, \frac{dx}{\nu(x)}\ .\]
Finally, the formulas \eqref{mueps} and \eqref{Rn} yield that for any $\varepsilon >0$ there is an $\eta >0$ fulfilling
\begin{align} \mathbf E\Big[\Big(R_n- \sum_{i=\sigma_n}^{\rho_n-1} f(X_i) \Big)^2\Big]  \le \varepsilon \Big(\int_{r_n}^{s_n} f(x) \, \frac{dx}{\nu(x)}\Big)^2 \ . 
\label{Rn3}
\end{align}

(iv) Putting pieces together, we obtain for given $\varepsilon >0$ and $\eta>0$ 
\begin{align*}
\mathbf P\Big(\Big| &\sum_{i=\sigma_n}^{\rho_n-1} f(X_i) - \int_{r_n}^{s_n} f(x) \, \frac{dx}{\nu(x)} \Big|\ge \varepsilon \int_{r_n}^{s_n} f(x) \, \frac{dx}{\nu(x)}   \Big)\\
&\le \mathbf P\Big(\Big|R_n- \sum_{i=\sigma_n}^{\rho_n-1} f(X_i) \Big|\ge \frac \varepsilon 2 \int_{r_n}^{s_n} f(x) \, \frac{dx}{\nu(x)}   \Big)\\
& \quad \mbox{} + \mathbf P\Big(\Big|R_n -\int_{r_n}^{s_n} f(x) \, \frac{dx}{\nu(x)} \Big|\ge \frac \varepsilon 2 \int_{r_n}^{s_n} f(x) \, \frac{dx}{\nu(x)} \ , \ \Delta_{i+1} \le \eta  X_i \text{ for all } i<\rho_n  \Big)\\
&\quad \mbox{} + \mathbf P\big(\Delta_{i+1} > \eta  X_i \text{ for some } i<\rho_n\ \big) \ .
\end{align*}
From \eqref{Rn3} for suitably chosen $\eta >0$ the first right-hand term becomes smaller than $\varepsilon$ and from \eqref{Rn} and \eqref{Rn2} the second one vanishes for large $n$. 
Consulting also Lemma \ref{Lmm3} (ii), we arrive at
\[ \mathbf P\Big(\Big| \sum_{i=\sigma_n}^{\rho_n-1} f(X_i) - \int_{r_n}^{s_n} f(x) \, \frac{dx}{\nu(x)} \Big|\ge \varepsilon \int_{r_n}^{s_n} f(x) \, \frac{dx}{\nu(x)}   \Big) \le \varepsilon \]
for $n$ large enough. This means that
\[ \sum_{i=\sigma_n}^{\rho_n-1} f(X_i) \stackrel P\sim  \int_{r_n}^{s_n} f(x) \, \frac{dx}{\nu(x)} \ . \]
To show $L_1$-convergence, it is by a convergence criterion due to F. Riesz  sufficient to have
\[ \mathbf E \Big[\sum_{i=\sigma_n}^{\rho_n-1} f(X_i)\Big]  \sim  \int_{r_n}^{s_n} f(x) \, \frac{dx}{\nu(x)}  \]
as $n \to \infty$. From  convergence in probability and Fatou's lemma we have
\[ \liminf_{n \to \infty} \mathbf E \Big[\sum_{i=\sigma_n}^{\rho_n-1} f(X_i)\Big]  \Big/\int_{r_n}^{s_n} f(x) \, \frac{dx}{\nu(x)} \ge 1 \ . \]
On the other hand,  \eqref{Rn} yields for given $\varepsilon >0$, suitable $\eta >0$ and large $n$
\[ \mathbf E \Big[\sum_{i=\sigma_n}^{\rho_n-1} f(X_i)\Big] = \mathbf E[R_n] \le (1+\varepsilon) \int_{r_n}^{s_n} f(x) \, \frac{dx}{\nu(x)}  \ . \]
This gives our first claim.

(v) For the second claim we again use the random variables $R_n$, now for  some $\eta>0$ (say $\eta=1/2$) with  $r_n=r:=2/\eta$. Recall that for this choice the $R_n$ are well-defined for all $n$. From inequality \eqref{Rn4}  we  obtain the estimate
\[ R_n \le (1-\eta)^{-\beta-3} \sup_{b \ge 2/\eta } \frac{\mu(b)}{\mu_\eta(b)} \int_2^{s_n}  f(x) \, \frac{dx}{\nu(x)} \ .\]
It follows
\begin{align*}
\mathbf E\Big[ \sum_{i=\sigma_n}^{\tau_n-1} f(X_i) \Big] \le \sum_{b < 2/\eta} f(b) + \mathbf E\Big[ \sum_{i=\sigma_n}^{\rho_n-1} f(X_i) \Big] = \sum_{b < 2/\eta} f(b)+\mathbf E[R_n] \ .
\end{align*}
Putting both estimates together and then replacing $\sigma_n$ by $\rho_n$ (which is just a change in notation),  we arrive at our second claim.
\end{proof}

\noindent
The next proposition presents a version of Proposition \ref{Prp1} in continuous time. As before, let $\tilde \rho_n$ and $\tilde\sigma_n$ denote the times, when the process $N_n$ falls below $r_n$ and $s_n$, respectively, while $\tilde \tau_n$ is the absorption time of $N_n$.

\begin{Proposition}   \label{Prp2}  Under the assumptions of Proposition \ref{Prp1},  as $n \to \infty$
\[ \int_{\tilde \sigma_n}^{\tilde \rho_n} f(N_n(t))\, dt  \ \stackrel 1\sim \int_{r_n}^{s_n} f(x) \, \frac {dx}{\mu(x)} \ . \]
Also,  as $n \to \infty$
\[ \mathbf E\Big[ \int_{\tilde \rho_n}^{\tilde \tau_n} f(N_n(t))\, dt \Big]=  O\Big( \int_{2}^{r_n} f(x) \, \frac {dx}{\mu(x)} \Big)\ . \]
\end{Proposition}

\paragraph{Example.} For $f(x)\equiv 1$ we obtain under the assumptions of Proposition \ref{Prp2}
\[ \tilde \rho_n -\tilde \sigma_n \stackrel 1 \sim \int_{r_n}^{s_n} \frac{dx}{\mu(x)} \ , \]
in particular,
\[ \tilde \rho_n \stackrel 1 \sim \int_{r_n}^{n} \frac{dx}{\mu(x)}  \]
as $n \to \infty$. \qed

\begin{proof}[Proof of Proposition \ref{Prp2}] We have
\[\int_{\tilde \sigma_n}^{\tilde \rho_n} f(N_n(t))\, dt = \sum_{i=\sigma_n}^{\rho_n-1} f(X_i)W_i  \]
and
\[\mathbf E\Big[ \sum_{i=\sigma_n}^{\rho_n-1} f(X_i)W_i \mid X\Big] = \sum_{i=\sigma_n}^{\rho_n-1} \frac{f(X_i)}{\lambda(X_i)} \ . \]
Thus for any $\eta >0$ by the Markov property
\[Ê\mathbf E \Big[ \Big(\int_{\tilde \sigma_n}^{\tilde \rho_n} f(N_n(t))\, dt -\sum_{i=\sigma_n}^{\rho_n-1} \frac{f(X_i)}{\lambda(X_i)}\Big)^2\Big] = \mathbf E \Big[ \sum_{i=\sigma_n}^{\rho_n-1} \frac{f(X_i)^2}{\lambda(X_i)^2} \Big]=\mathbf E [ R_n ]\ ,\]
where we now set
\[ R_n:= \sum_{i=\sigma_n}^{\rho_n-1} \frac{f(X_i)^2}{\lambda(X_i)^2} \frac{\Delta_{i+1} I_{\{\Delta_{i+1}\le \eta X_i\}}}{\nu_\eta(X_i)} \ . \]
Using \eqref{Rn} with $\varepsilon =1$, it follows
\[ \mathbf E \Big[ \Big(\int_{\tilde \sigma_n}^{\tilde \rho_n} f(N_n(t))\, dt -\sum_{i=\sigma_n}^{\rho_n-1} \frac{f(X_i)}{\lambda(X_i)}\Big)^2\Big] \le 2\int_{r_n}^{s_n} \frac{f(x)^2}{\lambda(x)^2\nu(x)} \, dx=  2 \int_{r_n}^{s_n} \frac{f(x)^2}{\lambda(x)\mu(x)} \, dx \ . \]
Furthermore, since we are in the dustless case, due to Lemma \ref{Lmm1} (iii) we have
\[ \mathbf E \Big[ \Big(\int_{\tilde \sigma_n}^{\tilde \rho_n} f(N_n(t))\, dt -\sum_{i=\sigma_n}^{\rho_n-1} \frac{f(X_i)}{\lambda(X_i)}\Big)^2\Big] \le o\Big(\int_{r_n}^{s_n} \frac{xf(x)^2}{\mu(x)^2} \, dx \Big) \ .\]
Letting $x_n$ be the point, where $xf(x)/\mu(x)$ takes its maximum in the interval $[r_n,s_n]$, we obtain in the same manner as in the preceding proof
\[ \mathbf E \Big[ \Big(\int_{\tilde \sigma_n}^{\tilde \rho_n} f(N_n(t))\, dt -\sum_{i=\sigma_n}^{\rho_n-1} \frac{f(X_i)}{\lambda(X_i)}\Big)^2\Big] = o\Big( \frac {x_nf(x_n)}{\mu(x_n)} \int_{r_n}^{s_n} \frac{f(x)}{\mu(x)} \, dx\Big) = o\Big( \Big(\int_{r_n}^{s_n} \frac{f(x)}{\mu(x)} \, dx\Big)^2\Big)\ .\]
On the other hand, Proposition \ref{Prp1} implies
\[ \sum_{i=\sigma_n}^{\rho_n-1} \frac{f(X_i)}{\lambda(X_i)} \stackrel 1 \sim \int_{r_n}^{s_n} \frac{f(x)}{\lambda(x)} \frac{dx}{\nu(x)} =\int_{r_n}^{s_n} \frac{f(x)}{\mu(x)} \, dx \ . \]
These last two formulas imply our first statement. Moreover
\[ \mathbf E\Big[ \int_{\tilde \rho_n}^{\tilde \tau_n} f(N_n(t))\, dt \Big] = \mathbf E \Big[ \sum_{i= \rho_n}^{\tau_n-1} \frac {f(X_i)}{\lambda(X_i)} \Big]\ , \]
therefore our second claim follows from the second statement of Proposition \ref{Prp1}.
\end{proof}

\noindent
Now we turn to the special case that $f(x)=1/x$, $x \ge 2$, where  Proposition \ref{Prp1} can be considerably extended. Here we content ourselves with the case $s_n=n$. Observe that the two next statements do not imply each other.

\begin{Proposition} \label{Prp3} Assume that the $\Lambda$-coalescent has no dust component.
 Let $2\le r_n\le    n$, $n \ge 1$, be a sequence of numbers fulfilling 
\begin{align*}  
r_n \le \gamma n 
\end{align*}
for all $n$ sufficiently large and some $\gamma <1$.  Also assume that
\begin{align*} \tilde \rho_n \stackrel P \to 0 
\end{align*}
as $n \to \infty$. Then
\[ \sum_{i=0}^{\rho_n-1} \frac 1{X_i} \stackrel 1\sim \log \frac{\kappa( n)}{\kappa(r_n)}  \]
and
\[\max_{1\le j \le \rho_n}\Big|\sum_{i=0}^{j-1} \frac 1{X_i} - \log \frac{\kappa(n)}{\kappa(X_j)}\Big| =o_P(1) \ .\]
\end{Proposition}

\begin{proof}
(i) The first statement is a special case of Proposition \ref{Prp1}:  due to Lemma \ref{Lmm1} (ii) we have  $1/(x\nu(x))= \lambda(x)/(x\mu(x))\sim \kappa'(x)/\kappa(x)$ as $x\to \infty$. Therefore
\[ \int_{r_n}^n \frac {dx}{x\nu(x)}  \sim \int_{r_n}^n \frac{\kappa'(x) \, dx}{\kappa(x)} = \log \frac {\kappa(n)}{\kappa(r_n)} \ . \]

(ii) For the proof of the second statement we proceed along similar lines as in the proof of Proposition \ref{Prp1}.  Using the second factorial moment of a binomial distribution, we have as a first estimate
\begin{align*}
\mathbf E\Big[ \frac{\Delta_{i+1}^2}{X_i^2} \mid X_i=b\Big]&= \frac 1{b^2\lambda(b)} \sum_{k=2}^b (k-1)^2\binom bk \int_{[0,1]} p^k (1-p)^{b-k} \frac{\Lambda(dp)}{p^2}\\
&\le \frac 1{b^2\lambda(b)} \int_{[0,1]} \sum_{k=0}^b k(k-1)\binom bk p^k (1-p)^{b-k} \frac{\Lambda(dp)}{p^2}\\
&= \frac 1{b^2\lambda(b)} \int_{[0,1]} p^2b(b-1) \frac{\Lambda(dp)}{p^2}\\
&\le \frac{\Lambda([0,1])}{ \lambda (b)} \ .
\end{align*}
This bound yields for large $n $
\begin{align*}
\mathbf E\Big[ \sum_{i=0}^{\rho_n-1}  \frac{\Delta_{i+1}^2}{X_i^2\nu(X_i)}\Big] &\le  \sum_{i=0}^{n-1} \mathbf E \Big[ \frac 1{\mu(X_i)} \, ; \, X_i > r_n\Big]\Lambda([0,1])\\
&=   \sum_{i=0}^{n-1} \mathbf E \Big[ \frac {\Delta_{i+1}}{\mu(X_i)\nu_{1/2}(X_i)}I_{\{\Delta_{i+1} \le X_i/2\}} \, ; \, X_i > r_n\Big]\Lambda([0,1])\\
&\le c  \mathbf E \Big[\sum_{i=0}^{\rho_n-1} \frac { \Delta_{i+1}}{\mu(X_i)\nu(X_i)}I_{\{\Delta_{i+1} \le X_i/2\}}\Big]
\end{align*}
with $c:=\Lambda([0,1]) \sup_{b \ge 4} \nu(b)/\nu_{1/2}(b) <\infty$. In view of Lemma \ref{Lmm1} (i) the function \[ x \mapsto (x-1)\mu(x)\nu(x)/x=\kappa(x)^2/(\lambda(x)/x(x-1))\] is  increasing, which entails $\Delta_{i+1}/( \mu(X_i)\nu(X_i)) \le \int_{X_{i+1}}^{X_i}  x((x-1)\mu(x)\nu(x))^{-1} \, dx$. By means of Lemma \ref{Lmm1} (ii) $\mu(x)\nu(x) \sim \kappa(x)^2/\kappa'(x)$, therefore
\begin{align*}
\mathbf E\Big[ \sum_{i=0}^{\rho_n-1}  &\frac{\Delta_{i+1}^2}{X_i^2\nu(X_i)}\Big] \le c   \int_{r_n/2}^n 
\frac{x\, dx}{ (x-1)\mu(x)\nu(x)}  \sim c   \int_{r_n/2}^n \frac {\kappa'(x)}{\kappa^2(x)} \, dx  \le  \frac{c}{\kappa(r_n/2)}  \ ,
\end{align*}
and consequently, since $\kappa(x)\to \infty$ in the dustless case,
\begin{align*}
\mathbf E\Big[ \sum_{i=0}^{\rho_n-1}  \frac{\Delta_{i+1}^2}{X_i^2\nu(X_i)}\Big]  = o(1) \ .
\end{align*}

(iii) Now we consider the martingale $M=(M_k)_{k \ge 0}$ given by $M_0=0$ a.s. and
\[ M_k := \sum_{i=0}^{k\wedge \tau_n-1} \Big( \frac {\Delta_{i+1}}{(X_i-1)\nu(X_i)}- \frac 1{X_i-1}  \Big) \ , \ k \ge 1 \ . \]
By means of the optional sampling theorem, since $\rho_n$ is a stopping time, we have, because $\nu(b) \ge 1$ for all $b \ge 2$,
\begin{align*}
\mathbf E \Big[ M_{\rho_n}^2\Big] \le \mathbf E\Big[ \sum_{i=0}^{\rho_n-1}  \frac{\Delta_{i+1}^2}{X_i^2\nu(X_i)^2}\Big] =o(1) \ .
\end{align*}
Thus by means of Doob's maximal inequality
\begin{align*} \max_{1\le j\le\rho_n}\Big| \sum_{i=0}^{j-1} \frac {\Delta_{i+1} }{(X_i-1)\nu(X_i)}-\sum_{i=0}^{j-1} \frac 1{X_i-1}  \Big|=o_P(1) \ . 
\end{align*}
Also for $j \le \rho_n$
\[ 0\le \sum_{i=0}^{j-1}  \frac 1{X_i-1}- \sum_{i=0}^{j-1} \frac 1{X_i} \le \sum_{m=X_{\rho_{n}-1}}^\infty \Big(\frac 1{m-1} - \frac 1m\Big) = \frac 1{X_{\rho_n-1}-1} \le \frac 1{r_n-1}\ , \]
and consequently
\begin{align} \max_{1\le j\le\rho_n}\Big| \sum_{i=0}^{j-1} \frac {\Delta_{i+1} }{(X_i-1)\nu(X_i)}-\sum_{i=0}^{j-1} \frac 1{X_i}  \Big|=o_P(1) \ . 
\label{Doob}
\end{align}

(iv) Note that in view of Lemma \ref{Lmm1} (i) the function $(x-1)\nu(x)= \kappa(x)/(\lambda(x)/x(x-1))$ is increasing and $\nu(x)/(x-1)^2= (x/(x-1))\cdot(\kappa(x)/(x-1))/\lambda(x)$ is decreasing. For $X_i\ge 2$ this yields  
\begin{align*}
0 &\le \frac 1{(X_{i+1}-1)\nu(X_{i+1})} - \frac 1{(X_i-1)\nu(X_i)} 
\\& = \frac {(X_{i+1}-1)^2}{\nu(X_{i+1})} \frac 1{(X_{i+1}-1)^3} - \frac {(X_{i}-1)^2}{\nu(X_i)} \frac 1{(X_{i}-1)^3}\\
&\le \frac {(X_{i}-1)^2}{\nu(X_i)} \Big(\frac 1{(X_{i+1}-1)^3}-\frac 1{(X_{i}-1)^3}\Big)\\
&= \frac{((X_i-1)^2+(X_{i}-1)(X_{i+1}-1)+(X_{i+1}-1)^2)\Delta_{i+1}}{\nu(X_i)(X_i-1)(X_{i+1}-1)^3}\\
&\le \frac {3(X_i-1)}{\nu(X_i)(X_{i+1}-1)^3} \Delta_{i+1}\\
&\le \frac {24 X_i}{\nu(X_i) X_{i+1}^3} \Delta_{i+1} \ .
\end{align*}
It follows
\begin{align*}
 \mathbf E\Big[ \max_{1\le j \le \rho_n}& \Big|\sum_{i=0}^{j-1} \frac {\Delta_{i+1} }{(X_i-1)\nu(X_i)}- \sum_{i=0}^{j-1}\frac {\Delta_{i+1}}{(X_{i+1}-1)\nu(X_{i+1})}\Big| \, ; \, \Delta_{i+1} \le X_i/2 \text{ for all } i < \rho_n \Big]\\
 &\le \mathbf E\Big[ \sum_{i=0}^{\rho_n-1} \Big|\frac {1}{(X_i-1)\nu(X_i)}-\frac {1}{(X_{i+1}-1)\nu(X_{i+1})}\Big|\Delta_{i+1}\, ; \, \Delta_{i+1} \le X_i/2 \text{ for all } i < \rho_n \Big]\\
 & \le \mathbf E \Big[ \sum_{i=0}^{\rho_n-1} \frac{8\cdot 24}{X_{i}^2 \nu(X_i)}Ê\Delta_{i+1}^2 \Big] = o(1) \ .
\end{align*}
Also, since $(x-1)\nu(x)$ is increasing
\[ \frac 1{(X_i-1)\nu(X_i)}\Delta_{i+1} \le \int_{X_{i+1}}^{X_i} \frac {dx}{(x-1)\nu(x)} \le \frac 1{(X_{i+1}-1)\nu(X_{i+1})}\Delta_{i+1}\ , \]
and we obtain 
\begin{align}\max_{1\le j \le \rho_n } \Big| \sum_{i=0}^{j-1} \frac 1{(X_i-1)\nu(X_i)}\Delta_{i+1}- \int_{X_{j}}^n \frac{dx} {(x-1)\nu(x)}\Big| = o_P(1) 
\label{Doob2}
\end{align}
on the event that $\Delta_{i+1} \le X_i/2 \text{ for all } i < \rho_n $. Lemma \ref{Lmm3} (ii) says that the complementary event has an asymptotically vanishing probability.

(v) Finally, by Lemma \ref{Lmm1} (ii) with $ r_n\le y \le n$ for $\chi < 1$
\begin{align*}
\int_{y}^n \frac {dx}{(x-1) \nu(x)} &= \int_{y}^n \frac {\lambda (x)/x^2 }{\mu(x)/x} \, dx + \int_y^n \frac{dx}{x(x-1)\nu(x)}\\&= \int_y^{n} \frac{\kappa'(x)}{\kappa(x)}\, dx + O\Big(\int_{r_n}^n \frac{\kappa'(x)\, dx}{\kappa(x)x^{\chi}}\Big)+ O(r_n^{-1})  \ ,
\end{align*}
and recalling $r_n \to \infty$
\[ \int_{r_n}^n \frac{\kappa'(x)\, dx}{\kappa(x)x^{\chi}}\sim \int_{r_n}^n \frac{  dx}{x\nu(x)x^{\chi}} \le \int_{r_n}^n \frac{dx}{x^{1+\chi}} =o(1) \ . \] 
Altogether  we obtain
\[ \max_{r_n \le y \le n} \Big|\int_{y}^n \frac {dx}{(x-1)\nu(x)} - \log \frac{\kappa(n)}{\kappa(y)} \Big| =o(1) \ .\]
Combining this formula with \eqref{Doob} and \eqref{Doob2} and recalling the definition of $\rho_n$,  we arrive at
\[ \max_{1 \le j \le \rho_n} \Big| \sum_{i=0}^{j-1}\frac 1{X_i} - \log \frac{\kappa(n)}{\kappa(X_j)} \Big| = o_P(1) \ . \]
This is our claim.
\end{proof}

\section{Proof of Theorem 1}

Under the assumptions of Proposition \ref{Prp2} we have
\begin{align}
\ell_n^*:=\int_0^{\tilde \rho_n} N_n(t)\, dt \stackrel 1\sim \int_{r_n}^n \frac x{\mu(x)} \, dx \ .
\label{totLaenge}
\end{align}
In particular, this formula holds for   $r_n:=cn$ with $0<c<1$ as anticipated in the introduction's remark. Here the assumptions of Proposition \ref{Prp2} are satisfied because of Lemma \ref{Lmm3} and
\[ \int_{cn}^n \frac {dx}{\mu(x)} \le (1-c) \frac n{\mu(cn)} =o(1)\ , \]
which in turn is valid in view of Lemma \ref{Lmm1} (iii).

In order to fill the gap up to $\ell_n$, we construct a distinguished sequence  of real numbers. 
We  construct the numbers $2 \le r_n\le n $, $n \ge 1$, satisfying 
\begin{align} \int_{r_n}^n \frac{dx}{\mu(x)} \to 0 \quad \text{and}\quad  \int_2^{r_n} \frac{x}{\mu(x)} \, dx = o \Big( \int_{r_n}^n \frac x{\mu(x)}\, dx\Big)  
\label{cn}
\end{align}
as $n \to \infty$. From Lemma \ref{Lmm3} we get that $\tilde \rho_n =o_P(1)$. Also, since by Lemma \ref{Lmm1} (i) $x/\mu(x)$ is decreasing
\[ \int_2^{r_n} \frac x{\mu(x)}\, dx \ge \frac{r_n(r_n-2)}{\mu(r_n)} \quad \text{and} \quad   \int_{r_n}^n \frac x{\mu(x)}\, dx \le \frac{r_n(n-r_n)} {\mu(r_n)}\ . \]
Therefore the second statement in \eqref{cn} entails $r_n -2= o(n-r_n)$ as $n \to \infty$ and consequently $r_n=o(n)$. Hence the sequence $r_n$, $n \ge 1$,  fulfils  all  requirements of Proposition \ref{Prp2}.

  For the construction of the numbers $r_n$ note that from Lemma \ref{Lmm1} (i) we have  $x/\mu(x) \ge 2/(\mu(2)(x-1))$ for $x\ge 2$ and consequently
\begin{align} 
\int_2^\infty \frac{x}{\mu(x)}\, dx = \infty \ .
\label{xmux}
\end{align}
We distinguish two cases. If $\int_2^\infty \frac {dx}{\mu(x)} <\infty$, then the required sequence is easily obtained, because the first condition of \eqref{cn} is fulfilled for any divergent sequence $r_n\le n$ and the second one by reason of \eqref{xmux}, if only $r_n$ is diverging slowly enough. 

Thus let us assume $\int_2^\infty \frac {dx}{\mu(x)} =\infty$ and  let $r_{n,m}\ge 2$ for given $m\in \mathbb N$ be the solution  of the equation
\[ \int_{r_{n,m}}^n  \frac{dx}{\mu(x)}=\frac 1m \ , \]
which  exists for $n\ge 3$ and $m\ge 1/\int_2^3  \frac{dx}{\mu(x)}$. Since $\int_2^\infty \frac {dx}{\mu(x)} =\infty$, we have $r_{n,m}\to \infty$ as $n \to \infty$. It follows 
\[ \int_{r_{n,m}}^n  \frac{x}{\mu(x)} \, dx \ge r_{n,m}/m  \ \text{ and }\ \int_2^{r_{n,m}} \frac{x}{\mu(x)} \, dx =o(r_{n,m})\] as $n\to \infty$ because of $x=o(\mu(x))$ from Lemma \ref{Lmm1} (iii). Therefore there are natural numbers $n_1<n_2 < \cdots$ such that 
\[ \int_2^{r_{n,m}} \frac{x}{\mu(x)} \, dx\le \frac 1m \int_{r_{n,m}}^n \frac{x}{\mu(x)} \, dx   \]
for all $n \ge n_m$. Now letting $r_n:= r_{n,m}$ for  $n=n_m, \ldots,n_{m+1}-1$, we obtain
\[ \int_{r_{n}}^n  \frac{dx}{\mu(x)}\le\frac 1m  \quad \text{and} \quad \int_2^{r_{n}} \frac{x}{\mu(x)} \, dx\le \frac 1m \int_{r_{n}}^n \frac{x}{\mu(x)} \, dx \]
for all $n\ge n_m$.  This implies \eqref{cn}.

Applying now Proposition \ref{Prp2} with $f(x)=x$, we obtain from \eqref{cn}
\[ \int_{0}^{\tilde \rho_n} N_n(t)\, dt \stackrel 1 \sim \int_{r_n}^n \frac{x}{\mu(x)} \, dx \sim \int_{2}^n \frac{x}{\mu(x)} \, dx\]
and
\[ \mathbf E\Big[ \int_ {\tilde \rho_n}^{\tilde \tau_n} N_n(t)\, dt \Big] = O\Big( \int_{2}^{r_n} \frac{x}{\mu(x)}\, dx\Big) =o \Big( \int_{2}^{n} \frac{x}{\mu(x)}\, dx\Big)Ê\ . \]
This implies our claim. \qed

\section{Proof of Theorem 2}

Again, let   $r_n$, $n \ge 1$, be any sequence fulfilling the assumptions of Proposition \ref{Prp1}. We investigate the lengths 
\begin{align} \widehat \ell_n^* := \int_0^{\tilde \rho_n} \widehat N_n(t)\, dt = \sum_{i=0}^{\rho_n-1} W_iY_i \ , 
\label{widehatell}
\end{align}
which are the total lengths of the external branches up to the $\rho_n$-th merger. Here $Y_i$, $i \ge 0$, denotes the number of external branches extant after the first $i$ merging events and $W_i$ as above the waiting time at the state $X_i$. 
In the proof we approximate $\widehat \ell_n^* $ by its conditional expectation given the block-counting process $N_n$, which in turn can be handled by means of Proposition \ref{Prp2} and 3.
We shall employ  the representation
\begin{align} Y_i = \sum_{k=1}^n I_{\{\zeta_k\ge i\}} 
\label{Yi}
\end{align}
where $\zeta_k$ denotes the number of  coalescent events before the $k$-th external branch  (out of $n$) merges with some other branches within the coalescent. 

\begin{Lemma}  \label{Lmm4} We have  for $i,j \ge 0$, $k,l=1, \ldots,n$  
\[ \mathbf P( \zeta_k \ge i \mid N_n) =  \frac{X_i-1}{n-1} \prod_{m=0}^{i-1} \Big( 1- \frac 1{X_m}\Big)\ \text{ a.s.} \]
and for $k \neq l$  
\[ \mathbf P( \zeta_k \ge i, \zeta_l \ge j \mid N_n) \le \mathbf P( \zeta_k \ge i \mid N_n)\mathbf P( \zeta_l \ge j \mid N_n) \ \text{ a.s.} \]
\end{Lemma}

\begin{proof} Let $A $ be a subset of $ \{1, \ldots,n\}$ with $a\ge 1$ elements, and let $\zeta_A$ be the number of mergers before one of the branches ending in $A$  gets involved into a merging event. Given $\Delta_1$ the first merger consists of a uniformly random choice of $\Delta_1+1$ members out of $X_0=n$ elements. Therefore we have    
\[ \mathbf P(\zeta_A \ge 1 \mid N_n) = \frac{ \binom {X_0-a}{\Delta_1+1}}{\binom {X_0}{\Delta_1+1}}= \frac{(X_0-a) \cdots (X_1-a)}{X_0\cdots X_1} = \frac{(X_1-1)\cdots (X_1-a)}{X_0\cdots (X_0-a+1)} \ \text{ a.s.}\]
or
\begin{align*}
\mathbf P(\zeta_A \ge 1 \mid N_n) = \frac{(X_1-1)\cdots (X_1-a)}{(X_0-1)\cdots(X_0-a)} \Big(1- \frac a{X_0}\Big)\ \text{ a.s.}
\end{align*}
Because of the Markov property  we may iterate this formula yielding
\begin{align} \mathbf P(\zeta_A \ge i \mid N_n) = \frac{(X_i-1)\cdots (X_i-a)}{(X_0-1)\cdots(X_0-a)}\prod_{m=0}^{i-1} \Big(1- \frac a{X_m}\Big) \ \text{ a.s.} 
\label{ANn}
\end{align}
In particular,   with $A=\{k\}$ and $a=1$ our first claim follows.

Similar for $k \neq l$ and $i \le j$  with $\zeta_{\{l\}}':= \zeta_{\{l\}}- \zeta_{\{k,l\}}$ and $N_{X_i}'(t):=N_n(t+W_0+\cdots+W_{i-1})$, $t \ge 0$, by means of the Markov property
\begin{align*}
\mathbf P( \zeta_k \ge i, \zeta_l \ge j \mid N_n) &= \mathbf P( \zeta_{\{k,l\}} \ge i \mid N_n) \mathbf P(\zeta_{\{l\}}' \ge j-i \mid N_{X_i}')\\&= \frac{(X_i-1)(X_i-2)}{(X_0-1)(X_0-2)} \prod_{m=0}^{i-1} \Big(1- \frac 2{X_m}\Big)\times \frac{X_j-1}{X_i-1} \prod_{m=i}^{j-1}\Big(1- \frac 1{X_m}\Big)\ \text{ a.s.}
\end{align*}
Since $X_i \le X_0$ and $(1- 2/X_m)\le (1-1/X_m)^2$, this implies
\begin{align*}
\mathbf P( \zeta_k \ge i, \zeta_l \ge j \mid N_n) &\le \frac{(X_i-1)^2}{(X_0-1)^2} \prod_{m=0}^{i-1} \Big(1- \frac 1{X_m}\Big)^2\times \frac{X_j-1}{X_i-1} \prod_{m=i}^{j-1}\Big(1- \frac 1{X_m}\Big)\\
&= \mathbf P( \zeta_k \ge i \mid N_n)\mathbf P( \zeta_l \ge j \mid N_n)  \  \text{ a.s.} \ ,
\end{align*}
which is our second claim.
\end{proof}

\begin{proof}[Proof of Theorem \ref{Thm2}] (i) First, we consider $\mathbf E[ \widehat \ell_n^* \mid N_n ] $. Due to \eqref{Yi} and Lemma \ref{Lmm4} we have
\begin{align} \mathbf E[ \widehat \ell_n^* \mid N_n ] &=  \sum_{i=0}^{\rho_n-1}\sum_{k=1}^n W_i \mathbf P(\zeta_k \ge i \mid N_n)\notag \\
&=  \frac n{n-1} \sum_{i=0}^{\rho_n-1} W_i(X_i-1)\prod_{m=0}^{i-1}\Big(1-\frac 1{X_m}\Big) \ \text{ a.s.}
\label{intcond}
\end{align}
Since $\sum_{m=0}^{i-1} X_m^{-2} \le \sum_{a=X_{i-1}}^\infty a^{-2} \le (X_{i-1}-1)^{-1}\le (r_n-1)^{-1}$ for $i\le \rho_n$ and in view of Proposition \ref{Prp3},
\begin{align*}
\prod_{m=0}^{i-1}\Big(1-\frac 1{X_m}\Big)= \exp \Big( - \sum_{m=0}^{i-1} \frac 1{X_m} + O(r_n^{-1}) \Big) = \frac{\kappa(X_i)}{\kappa(n)} \exp( o_P(1)) \ ,
\end{align*}
where the $o_P(1)$ may be taken uniformly in $i < \rho_n$ in the sense of Proposition \ref{Prp3}. Thus we obtain
\[ \mathbf E[ \widehat \ell_n^* \mid N_n ] \stackrel P \sim \frac 1{\kappa(n)}\sum_{i=0}^{\rho_n-1} W_i(X_i-1) \kappa(X_i) = \frac 1{\kappa(n)} \int_0^{\tilde \rho_n} f(N_n(t))\, dt   \]
with $f(x):= (x-1)\kappa(x)$. This function satisfies the assumption of Proposition \ref{Prp2}. Because of $f(x) \sim \mu(x)$, $r_n\to \infty$  we obtain
\begin{align}
\mathbf E[ \widehat \ell_n^* \mid N_n ] \stackrel P \sim \frac 1{\kappa(n)}\int_{r_n}^n (x-1)\kappa(x) \frac{dx}{\mu(x)} \sim \frac{ n-r_n}{\kappa(n)}= \frac {n(n-r_n)}{\mu(n)} \ .
\label{ellcond}
\end{align}

(ii)  Next, we have
\begin{align*}
\mathbf E\big[ ( \widehat\ell_n ^* & - \mathbf E[\widehat\ell_n^* \mid N_n])^2 \mid N_n] \\
&= \mathbf E\Big[ \Big( \sum_{i=0}^{\rho_n-1}\sum_{k=1}^n\Big(W_i I_{\{\zeta_k\ge i\}} - W_i \mathbf P(\zeta_k\ge i\mid N_n) \Big)\Big)^2 \mid N_n\Big]\\
&= \sum_{i,j=0}^{\rho_n-1}\sum_{k,l=1}^n W_iW_j\big(\mathbf P(\zeta_k \ge i, \zeta_l \ge j\mid N_n)- \mathbf P(\zeta_k \ge i\mid N_n)\mathbf P( \zeta_l \ge j\mid N_n) \big) \text{ a.s.}
\end{align*}
Applying Lemma \ref{Lmm4}, it follows
\begin{align*}
\mathbf E\big[ ( \widehat\ell_n ^*  - \mathbf E[\widehat\ell_n^* \mid N_n])^2 \mid N_n] &\le \sum_{i,j=0}^{\rho_n-1}\sum_{k=1}^n W_iW_j\mathbf P(\zeta_k \ge i\vee  j\mid N_n)\\
&\le\sum_{i,j=0}^{\rho_n-1}  W_iW_j \sum_{k=1}^n \mathbf P(\zeta_k \ge i \mid N_n)\\
&= \mathbf E[\widehat\ell_n^* \mid N_n] \sum_{j=0}^{\rho_n-1} W_j\\
&= \tilde \rho_n \mathbf E[\widehat\ell_n^* \mid N_n] \text{ a.s.} 
\end{align*}
Since by assumption $\tilde \rho_n =o_P(1)$ and $r_n \le \gamma n$,   \eqref{ellcond} yields
\begin{align*}\mathbf E\big[ ( \widehat\ell_n ^*  - \mathbf E[\widehat\ell_n^* \mid N_n])^2 \mid N_n] = o_P\Big(\frac{n(n-r_n)}{\mu(n)}\Big)=o_P\Big(\frac{n^2}{\mu(n)}\Big) \ .
\end{align*}
Because of Lemma \ref{Lmm1} (i) $n(n-1)/\mu(n)$ is increasing, which implies
\begin{align} \widehat\ell_n ^*  - \mathbf E[\widehat\ell_n^* \mid N_n]=o_P\Big(\frac{n}{\sqrt{\mu(n)}}\Big)= o_P\Big(\frac{n^2}{\mu(n)}\Big) 
\label{est3}
\end{align}
and because of \eqref{ellcond}
\begin{align}
\widehat\ell_n ^* \stackrel P\sim \frac{n(n-r_n)}{\mu(n)} \ .
\label{remark}
\end{align}
In particular, as discussed in the proof of Theorem \ref{Thm1} and addressed in the introduction, this approximation is valid for the sequence $r_n=cn$ with $0<c<1$.

(iii) Finally,  let us switch to the numbers $r_n$, $n \ge 1$, constructed in the proof of Theorem \ref{Thm1} and  fulfilling \eqref{cn} as well as  $r_n=o(n)$. As above
\begin{align*}
\mathbf E[ \widehat \ell_n -  \widehat \ell_n^*\mid N_n]  &=  \frac n{n-1}\sum_{i=\rho_n}^{\tau_n-1} W_i(X_i-1) \prod_{m=0}^{i-1} \Big( 1- \frac 1{X_m}\Big)\\
& \le 2\exp \Big( - \sum_{m=0}^{\rho_n-1} \frac 1{X_m}\Big)\sum_{i=\rho_n}^{\tau_n-1} W_iX_i \ \text{ a.s.}
\end{align*}
From   the Markov property and Theorem \ref{Thm1}, applied to the coalescent with initial value $X_{\rho_n}$, we obtain
\[ \sum_{i=\rho_n}^{\tau_n-1} W_iX_i \stackrel P\sim \int_2^{X_{\rho_n}} \frac{x}{\mu(x)}\, dx \le \int_2^{r_n} \frac{x}{\mu(x)}\, dx\ . \]
By \eqref{cn} and by monotonicity of $x/\mu(x)$ it follows
\[ \sum_{i=\rho_n}^{\tau_n-1} W_iX_i = o_P\Big(\int_{r_n}^n \frac{x}{\mu(x)}\, dx\Big) = o_P\Big( \frac {r_n}{\mu(r_n)} n\Big) \ .\]
Moreover, from Proposition \ref{Prp2}
\[\exp \Big( - \sum_{m=0}^{\rho_n-1} \frac 1{X_m}\Big) \stackrel P\sim \frac{\mu(r_n)}{r_n} \frac{n}{\mu(n)} \ ,\]
and we arrive at
\[ \mathbf E[\widehat  \ell_n -  \widehat \ell_n^*\mid N_n]  = o_P\Big( \frac{n^2}{\mu(n)} \Big)\ . \]
Hence
\begin{align*}\widehat  \ell_n - \widehat \ell_n^*  = o_P\Big( \frac{n^2}{\mu(n)} \Big)\ . 
\end{align*}
This estimate in combination with \eqref{remark} and $r_n=o(n)$ proves our theorem.
\end{proof}

\section{Proof of Theorem 3}

(i) As to the second claim, 
\[ \int_2^n \Big(\frac x{\mu(x)}- \frac n{\mu(n)}\Big)\, dx= \int_2^n \int_x^n  \frac{\kappa'(y)}{\kappa(y)^2} \, dy\, dx = \int_2^n (y-2) \frac{\kappa'(y)}{\kappa(y)^2} \, dy \ , \] 
and by Lemma \ref{Lmm1} (ii) and Lemma \ref{Lmm2}
\[ \int_2^n \Big(\frac x{\mu(x)}- \frac n{\mu(n)}\Big)\, dx \sim \int_2^n y \frac{\lambda(y)}{\mu(y)^2}\, dy \sim \int_2^n \frac{L(y)}{L^*(y)^2}\ dy \sim \frac{nL(n)}{L^*(n)^2} \ . \]

(ii) Turning to the first claim,  we now strive for a lower bound for $\widecheck \ell_n$.    We   resort to the definitions \eqref{totLaenge} and \eqref{widehatell} and set
\[\widecheck \ell_n^*:= \ell_n^* - \widehat \ell_n^*\ .\] 
Again, we first investigate  its conditional expectation, given $N_n$. Note that $\mathbf E[ \ell_n^* \mid N_n]= \ell_n^*  $ a.s., therefore in view of \eqref{est3}  
\begin{align*}
\widecheck \ell_n^* - \mathbf E[\widecheck \ell_n^* \mid N_n] =   \mathbf E[\widehat\ell_n^* \mid N_n] - \widehat\ell_n^* = o_P\Big(\frac n{\sqrt{\mu(n)}}\Big)  \ .
\end{align*}
Lemma \ref{Lmm1} (i) implies $\mu(n) \ge n\mu(2)/2$ for $n\ge 2$, hence
\begin{align}
\widecheck \ell_n^* - \mathbf E[\widecheck \ell_n^* \mid N_n] =o_P(n^{1/2} )\ . 
\label{ellwidecheck}
\end{align}

We like to estimate $\mathbf E[\widecheck \ell_n^* \mid N_n] $ from below. For this purpose we specify our choice of the numbers $r_n$. 
 We fix $h \in \mathbb N$, we define   stopping times $0=\rho_{n,h}\le \rho_{n,h-1}\le \cdots \le \rho_{n,1}$  and the corresponding times  $\tilde \rho_{n,g}$ as
\begin{align} \rho_{n,g} := \min \{ i \ge 0: X_i \le   \tfrac gh n\}  \ , \ \tilde \rho_{n,g}:= \inf \{ t\ge 0: N_n(t) \le \tfrac gh n \} \ ,  \ g=1, \ldots,h  \ , \ .
\label{rhonl}
\end{align}
As we already argued in the proof of Theorem \ref{Thm1}, we may apply Propositions 2 and 3 to these stopping times. 
We proceed now in the same manner  as in \eqref{intcond}. Using $X_i\le n$ respectively $X_i \ge n(X_i-1)/(n-1)$, we obtain
\begin{align*} \mathbf E\Big[ \sum_{i=\rho_{n,g}}^{\rho_{n,g-1}-1}& W_i(X_i-Y_i) \mid N_n\Big] = \sum_{i=\rho_{n,g}}^{\rho_{n,g-1}-1} W_i\Big(X_i- \frac n{n-1} (X_i-1) \prod_{m=0}^{i-1} \Big(1- \frac 1{X_m}\Big)\Big)\\
&\ge  \sum_{i=\rho_{n,g}}^{\rho_{n,g-1}-1} W_iX_i\Big(1- \prod_{m=0}^{i-1} \Big(1- \frac 1{X_m}\Big)\Big) \\
&\ge\Big(1- \prod_{m=0}^{\rho_{n,g}-1} \Big(1- \frac 1{X_m}\Big)\Big) \sum_{i=\rho_{n,g}}^{\rho_{n,g-1}-1} W_iX_i \\
&\ge \Big(1- \exp \Big(- \sum_{m=0}^{\rho_{n,g}-1} \frac 1{X_m}\Big)\Big)\int_{\tilde \rho_{n,g}}^{\tilde \rho_{n,g-1}} N_n(t)\, dt \quad \text { a.s.}
\end{align*}
Proposition \ref{Prp2} together with Lemma \ref{Lmm2} (ii) implies
\[ \int_{\tilde \rho_{n,g}}^{\tilde \rho_{n,g-1}} N_n(t)\, dt \stackrel 1\sim  \int_{(g-1)n/h}^{ gn/h} \frac x{\mu(x)} \, dx \sim \frac{n}{hL^*(n)} \ , \]
and Proposition \ref{Prp3} yields
\[ \sum_{m=0}^{\rho_{n,g}-1}\frac 1{X_m} \stackrel 1\sim \log \frac{\kappa(n)}{\kappa(gn/h)}\sim \log \frac{L^*(n)}{L^*(gn/h)}\ . \]
Because $L^*$ is slowly varying and because of \eqref{Lstern}, we have
\begin{align*}
\log \frac{L^*(n)}{L^*(gn/h)} \sim \frac {L^*(n)- L^*(gn/h)}{L^*(gn/h)} \sim \frac {L(n)}{L^*(n)} \log \frac hg \ ,
\end{align*}
hence because of Lemma \ref{Lmm2} (ii)
\begin{align}
\sum_{m=0}^{\rho_{n,g}-1}\frac 1{X_m}  \stackrel 1\sim \frac {L(n)}{L^*(n)} \log \frac hg =o_P(1)
\label{1X_m}
\end{align}
and
\begin{align*}
1- \exp \Big(- \sum_{m=0}^{\rho_{n,g}-1} \frac 1{X_m}\Big) \stackrel P\sim \frac {L(n)}{L^*(n)} \log \frac hg  \ .
\end{align*}
Altogether
\begin{align*} 
\mathbf E\Big[ \sum_{i=\rho_{n,g}}^{\rho_{n,g-1}-1} W_i(X_i-Y_i) \mid N_n\Big] \ge (1+o_P(1)) \frac{nL(n)}{L^*(n)^2} \frac 1h \log \frac hg   
\end{align*}
and consequently 
\begin{align*}
\mathbf E [ \widecheck \ell_n^* \mid N_n] &= \mathbf E\Big[ \sum_{i=0}^{\rho_{n,1}-1} W_i(X_i-Y_i) \mid N_n\Big] \\
&\ge (1+o_P(1))\frac{nL(n)}{L^*(n)^2} \sum_{g=2}^{h} \frac 1h \log \frac hg\\
&\ge (1+o_P(1))\frac{nL(n)}{L^*(n)^2} \int_{2/h}^1 \log \frac 1z \, dz \\
&= (1+o_P(1))\frac{nL(n)}{L^*(n)^2} \Big(1- \frac 2h - \frac 2h\log \frac 2h\Big)\ .
\end{align*}
In view  of \eqref{ellwidecheck} this estimate transfers to $\widecheck \ell_n^* $. We have $\widecheck \ell_n\ge \widecheck \ell_n^*$, thus letting $h \to \infty$,  we obtain 
\begin{align}     \widecheck \ell_n   \ge (1+o_P(1))\frac{nL(n)}{L^*(n)^2}
\label{estim5}
\end{align}
as $ n \to \infty$.

(iii) Coming to an upper bound, we have in view of $  \prod_{m=0}^{i-1} \big(1- 1/X_m\big) \ge 1- \sum_{m=0}^{i-1} 1/X_m$
\begin{align}
\mathbf E [\widecheck \ell_n] &= \mathbf E \big[\mathbf E[\widecheck \ell_n \mid N_n]\big]\notag \\
&= \mathbf E\Big[\sum_{i=0}^{\tau_n-1}  W_i\Big(X_i- \frac n{n-1} (X_i-1)\prod_{m=0}^{i-1} \Big(1- \frac 1{X_m}\Big)\Big)\Big]\notag \\
&\le \mathbf E\Big[ \sum_{i=0}^{\tau_n-1}W_i\Big(1+ \frac n{n-1}X_i \sum_{m=0}^{i-1} \frac 1{X_m}\Big) \Big]\ .
\label{estim1}
\end{align}
From Proposition \ref{Prp2} and since $x=o(\mu(x))$ (see Lemma \ref{Lmm1} (iii)),
\begin{align}\mathbf E\Big[ \sum_{i=0}^{\tau_n-1}W_i\Big] = \mathbf E[ \tilde \tau_n]= O\Big( \int_2^n \frac{dx}{\mu(x)}\Big)= o\Big(\int_2^n \frac{dx}{x}\Big) =o(\log n) \ . 
\label{estim2}
\end{align}
Furthermore, by means of the Markov property
\begin{align*}
\mathbf E\Big[ \sum_{i=0}^{\tau_n-1}W_iX_i \sum_{m=0}^{i-1} \frac 1{X_m} \Big] &=\sum_{i=0}^{n-1}\mathbf E\Big[ \frac{X_i}{\lambda(X_i)} \sum_{m=0}^{i-1} \frac 1{X_m} \, ; \, X_i>2 \Big]\\
&= \sum_{i=0}^{n-1}\mathbf E\Big[ \frac{X_i \Delta_{i+1}}{\mu(X_i)}  \sum_{m=0}^{i-1}  \frac 1{X_m}\, ; \, X_i > 2 \Big]\\
&=\mathbf E \Big[ \sum_{m=0}^{\tau_n-1} \frac 1{X_m} \sum_{i=m+1}^{\tau_n-1} \frac{X_i \Delta_{i+1}}{\mu(X_i)}\Big] \ .
\end{align*}
Since $x/\mu(x)$ is decreasing, we have $\sum_{i=m+1}^{\tau_n-1} X_i \Delta_{i+1}/\mu(X_i) \le \int_1^{X_{m+1}} x/\mu(x)\, dx$, where we let $x/\mu(x):=2/\mu(2)$ for $1\le x \le 2$. Thus
\[\mathbf E\Big[ \sum_{i=0}^{\tau_n-1}W_iX_i \sum_{m=0}^{i-1} \frac 1{X_m} \Big]  \le \mathbf E\Big[ \sum_{m=0}^{\tau_n-1} \frac 1{X_m} \int^{X_{m+1}}_1 \frac x{\mu(x)}\, dx \Big]\ .\]
Invoking the Markov property once again, we obtain
\[ \mathbf E\Big[ \sum_{i=0}^{\tau_n-1}W_iX_i \sum_{m=0}^{i-1} \frac 1{X_m} \Big] \le  \mathbf E\Big[ \sum_{m=0}^{\tau_n-1} \frac {\Delta_{m+1}}{(X_m-1)\nu(X_m) } \int^{X_{m+1}}_1 \frac x{\mu(x)}\, dx \Big]\ , \]
and taking now into account that the functions $(x-1)\nu(x)$  and $\int_1^x z/\mu(z)\, dz$ are increasing, we end up with
\begin{align} \mathbf E\Big[ \sum_{i=0}^{\tau_n-1}W_iX_i \sum_{m=0}^{i-1} \frac 1{X_m} \Big]  \le \int_1^n \frac 1{(x-1)\nu(x)} \int_1^x \frac z{\mu(z)} \ dz\, dx \ , 
\label{estim3}
\end{align}
where we let $(x-1)\nu(x):=\nu(2)$ for $1\le x \le 2$. Using Lemma \ref{Lmm2} (ii)  it follows
\begin{align*}
 \int_1^n \frac 1{(x-1)\nu(x)} \int_1^x \frac z{\mu(z)} \ dz\, dx & \sim  \int_2^n \frac {L(x)}{xL^*(x)} \int_2^x \frac 1{L^*(z)} \ dz\, dx \\
 &\sim \int_2^n  \frac {L(x)}{xL^*(x)} \frac{x}{L^*(x)}  \, dx \\
 &\sim \frac{nL(n)}{L^*(n)^2}   \ .
\end{align*}
Together with \eqref{estim1}, \eqref{estim2} and \eqref{estim3} we obtain all in all
\[ \mathbf E [\widecheck \ell_n]  \le (1+o(1))\frac{nL(n)}{L^*(n)^2} \ . \]
Combining this result with \eqref{estim5}, we get 
\[ \widecheck \ell_n \stackrel P\sim \frac{nL(n)}{L^*(n)^2}\ , \]
and invoking once again the convergence criterion of F. Riesz, also $L_1$-convergence follows. This finishes the proof of Theorem \ref{Thm3}. \qed

\section{Proof of Theorem 4}

(i) Concerning  the first claim, we have from Theorem \ref{Thm2} and Lemma \ref{Lmm2} (ii)
\[  \widehat \ell_{n,1} = \widehat \ell_n \stackrel P\sim \frac{n^2}{\mu(n)} \sim \frac n{L^*(n)}\ ,  \]
where $L^*$ is slowly varying. To obtain also $L_1$-convergence, we use Theorem \ref{Thm1} saying that
\[ \ell_n \stackrel 1\sim \int_2^n \frac{x}{\mu(x)}\, dx \sim \int_2^n \frac{dx}{L^*(x)} \sim \frac n{L^*(n)} \ . \]
In particular the random variables $L^*(n)\ell_n/n$, $n \ge 1$, make a uniformly integrable sequence. Since $\widehat \ell_{n,1}\le \ell_n$, this holds true also for the random variables $L^*(n)\widehat \ell_{n,1}/n$, $n \ge 1$. Therefore the above convergence in probability converts to $L_1$-convergence.

(ii) Now we turn to the case $a\ge 2$. Preliminary we study  lengths of the form
\[ \widehat \ell_{n,a}^* = \int_{\tilde \sigma_n}^{\tilde \rho_n} \widehat N_{n,a}(t)\, dt = \sum_{i=\sigma_n}^{\rho_n-1} W_iY_{i,a} \ , \]
where  $Y_{i,a}$ denotes the number of internal branches of order $a$ present in the coalescent after $i$ merging events. We are going to bound these numbers from below. Let $A$ denote a subset of $\{1, \ldots,n\}$ with $a$ elements.   For $1 \le k \le i$ let  
$E_{i,k,A}$ be the event that the external branches ending in $A$ are not involved in the first $k-1$ mergers, next coalesce  with the $k$-th merger to one lineage without any other branch participating, and then remain untouched by merging events till the $i$-th merger. These are disjoint events, which all contribute to $Y_{i,a}$, therefore
\[ Y_{i,a} \ge Y_{i,a}':= \sum_{k=1}^{i} \sum_{A}I_{E_{i,k,A}} \ , \]
where the second sum is taken over all $A\subset \{1, \ldots,n\}$ with $a$ elements. 

\begin{Lemma} \label{Lmm5}   Let both $A,A' \subset \{1,\ldots,n\}$ have $a\ge 1$ elements. Then for $ 1 \le k \le i$ 
\[ \mathbf P(E_{i,k,A} \mid N_n)=\frac{a!}{(X_0-1)\cdots(X_0-a)} \frac{X_{i}}{X_{k-1}}\prod_{m=0}^{k-2}\Big(1-\frac a{X_m}\Big) \prod_{m=k}^{i} \Big( 1- \frac 1{X_{m}}\Big)I_{\{\Delta_k=a-1\}} \text{ a.s.}\]
and for $A\neq A'$ or $k\neq l$ and for $1 \le j \le \ell$
\begin{align*}
\mathbf P(E_{i,k,A}\cap E_{j,l,A'}\mid N_n) \le (1+ O(X_{k-1}^{-1})+O(X_{l-1}^{-1}))\mathbf P(E_{i,k,A}\mid N_n)\mathbf P(E_{j,l,A'}\mid N_n) \text{ a.s.}
\end{align*}
\end{Lemma}

\begin{proof}
We proceed similar as in the proof of Lemma \ref{Lmm4}. From \eqref{ANn} and  the Markov property
\begin{align*}
\mathbf P(E_{i,k,A} \mid N_n) = &\frac{(X_{k-1}-1)\cdots (X_{k-1}-a)}{(X_0-1)\cdots(X_0-a)} \prod_{m=0}^{k-2}\Big(1-\frac a{X_m}\Big) \\
&\quad \mbox{}\times \binom {X_{k-1}}a ^{-1}I_{\{\Delta_k=a-1\}} \times  \frac{X_{i}-1}{X_{k}-1}\prod_{m=k}^{i-1} \Big( 1- \frac 1{X_{m}}\Big) \text{ a.s.}
\end{align*}
Note that $X_k-1= X_{k-1}-a$ on the event $\{\Delta_k=a-1\}$. Thus all factors containing $X_{k-1}$ cancel up to the term $X_{k-1}$ in the denominator. Replacing also $X_i-1$ by $X_i(1- 1/X_i)$, the first statement follows.

For the second claim note that in the case $A\neq A'$ and  $A\cap A'\neq \emptyset$ or in the case $A=A'$ and $k\neq l$ the events $E_{i,k,A}$ and $ E_{j,l,A'}$ are disjoint,  then our claim is obvious. Therefore we may assume that $A\cap A'=\emptyset$. Let us consider the case $k < l < i\le j$. Then from \eqref{ANn} and  the Markov property
\begin{align*}
\mathbf P(E_{i,k,A}&\cap E_{j,l,A'}\mid N_n) =  \frac{(X_{k-1}-1)\cdots (X_{k-1}-2a)}{(X_0-1)\cdots(X_0-2a)} \prod_{m=0}^{k-2}\Big(1-\frac {2a}{X_m}\Big)\\
&\quad \mbox{}\times \binom {X_{k-1}}a ^{-1}I_{\{\Delta_k=a-1\}}\times \frac{(X_{l-1}-1)\cdots (X_{l-1}-a-1)}{(X_k-1)\cdots(X_k-a-1)} \prod_{m=k}^{l-2}\Big(1-\frac {a+1}{X_m}\Big)\\
&\quad \mbox{}\times \binom {X_{l-1}}a ^{-1}I_{\{\Delta_l=a-1\}}\times \frac{(X_i-1)(X_i-2)}{(X_l-1)(X_{l}-2)}\prod_{m=l}^{i-1}\Big(1-\frac {2}{X_m}\Big)\\
&\quad \mbox{}\times \frac{X_j-1}{X_i-1 }\prod_{m=i}^{j-1}\Big(1-\frac {1}{X_m}\Big) \ \text{ a.s.}
\end{align*}
Here  the product $(X_k-1)\cdots(X_k-a-1)= (X_{k-1}-a)\cdots(X_{k-1}-2a)$  cancels out on the event $\{\Delta_k=a-1\}$, and again only the factor $X_{k-1}$ remains. Similarly $X_l-2$ and $X_{l-1}-a-1$ cancel. By increasing some other terms, we get
\begin{align*}
\mathbf P(E_{i,k,A}&\cap E_{j,l,A'}\mid N_n) \le \frac{a!a!}{(X_0-1)\cdots(X_0-2a)} \prod_{m=0}^{k-2}\Big(1-\frac {a}{X_m}\Big)^2\\
&\quad \mbox{}\times I_{\{\Delta_k=a-1\}}\frac 1{X_{k-1}}\prod_{m=k}^{l-2}\Big(1-\frac {a}{X_m}\Big)\Big(1-\frac {1}{X_m}\Big)\\
&\quad \mbox{}\times I_{\{\Delta_l=a-1\}}\frac 1{X_{l-1}}\prod_{m=l}^{i-1}\Big(1-\frac {1}{X_m}\Big)\Big(1-\frac {1}{X_m}\Big)\\
&\quad \mbox{}\times (X_i-1)(X_j-1) \prod_{m=i}^{j-1}\Big(1-\frac {1}{X_m}\Big) \text{ a.s.}
\end{align*}
Replacing also $X_i-1$ and $X_j-1$ as above and comparing with our first formula, we obtain
\begin{align*}
\mathbf P(E_{i,k,A}&\cap E_{j,l,A'}\mid N_n) \\
&\le \frac{(X_0-1)\cdots (X_0-a)}{(X_0-a-1)\cdots (X_0-2a)} \Big(1- \frac a{X_{k-1}}\Big)^{-1} \mathbf P(E_{i,k,A}\mid N_n)\mathbf P(E_{j,l,A'}\mid N_n) \text{ a.s.}
\end{align*}
This implies our second claim. Other cases like $k<i<l<j$ are treated similarly.
\end{proof}

\newpage

(iii) Coming back to the proof of the theorem's second claim, we first consider, similarly as above, conditional expectations given $N_n$. Recall from \eqref{rhonl} the definition of $\rho_{n,g}$. We have from Lemma \ref{Lmm5} for $g \ge 2$
\begin{align*}
\mathbf E\Big[ &\sum_{i=\rho_{n,g}}^{\rho_{n,g-1}-1} W_iY_{i,a}'  \mid N_n\Big] = \sum_{i=\rho_{n,g}}^{\rho_{n,g-1}-1}W_i\sum_{k=1}^{i} \sum_A \mathbf P(E_{i,k,A} \mid N_n)\\
&= \frac {X_0}{X_0-a} \sum_{i=\rho_{n,g}}^{\rho_{n,g-1}-1}W_i\sum_{k=1}^{i}  \frac{X_{i}}{X_{k-1}}\prod_{m=0}^{k-2}\Big(1-\frac a{X_m}\Big) \prod_{m=k}^{i} \Big( 1- \frac 1{X_{m}}\Big)I_{\{\Delta_k=a-1\}} \text{ a.s.}
\end{align*}
The products  may be estimated from above by 1 and, by means of Proposition \ref{Prp3} and the bound $(1-z)\ge \exp(-cz)$ for $z \le 1/2$ and a suitable $c>0$, from below uniformly in $i,k$ by
\[ \prod_{m=0}^{\rho_{n,1}-1}\Big(1-\frac a{X_m}\Big)\ge \exp\Big(-ca \sum_{m=0}^{\rho_{n,1}-1} \frac 1{X_m} \Big)\stackrel P \sim \Big(\frac{\kappa(n/h)}{\kappa(n)}\Big)^{ca}\sim  \Big(\frac{L^*(n/h)}{L^*(n)}\Big)^{ca} \sim 1\ . \]
Consequently we may replace the products by 1 and obtain
\begin{align}
\mathbf E\Big[ \sum_{i=\rho_{n,g}}^{\rho_{n,g-1}-1}W_i Y_{i,a}'  \mid N_n\Big] &\stackrel P\sim  \sum_{i=\rho_{n,g}}^{\rho_{n,g-1}-1}W_iX_{i}\sum_{k=1}^{i} \frac{1}{X_{k-1}}I_{\{\Delta_k=a-1\}}
\label{Yia}
\end{align}
as $n \to \infty$. 

From Lemma \ref{Lmm2} (i) and \eqref{1X_m},
\[\sum_{m=0}^{\rho_{n,g}-1} \frac 1{X_m} \mathbf P(\Delta_{m+1}=a-1) \stackrel P\sim \frac 1{(a-1)a}\sum_{m=0}^{\rho_{n,g}-1} \frac 1{X_m} \stackrel P\sim  \frac 1{(a-1)a}\frac {L(n)}{L^*(n)} \log \frac hg \ .\]
By the Markov property
\begin{align*}\mathbf E\Big[\Big( &\sum_{m=0}^{\rho_{n,g}-1}\frac{1}{X_{m}} \big(I_{\{\Delta_{m+1}=a-1\}} - \mathbf P(\Delta_{m+1}=a-1)\big) \Big)^2 \Big] \\& = \mathbf E\Big[ \sum_{m=0}^{\rho_{n,g}-1}\frac{1}{X_{m}^2} \big(I_{\{\Delta_{m+1}=a-1\}} - \mathbf P(\Delta_{m+1}=a-1)\big)^2  \Big] \\
&\le \mathbf E\Big[ \frac1{X_{\rho_{n,g }-1}-1} \Big] \le \frac {1}{ (gn/h)-1} \ .
\end{align*}
Therefore
\[ \sum_{m=0}^{\rho_{n,g}-1}\frac 1{X_m}I_{\{\Delta_{m+1}=a-1\}}\stackrel P\sim   \frac 1{a(a-1)}\frac {L(n)}{L^*(n)} \log \frac hg \ ,\]
also once more by Proposition \ref{Prp2}  and Lemma \ref{Lmm2} (ii)
\[ \sum_{i=\rho_{n,g}}^{\rho_{n,g-1}-1}W_iX_{i} = \int_{\tilde\rho_{n,g}}^{\tilde \rho_{n,g-1}} N_n(t)\, dt  \stackrel 1\sim  \int_{(g-1)n/h}^{ gn/h} \frac x{\mu(x)} \, dx \sim \frac{n}{hL^*(n)} \ .\]
These formulas yield together with \eqref{Yia} the following lower and upper bounds for $g \ge 2$
\begin{align}
\mathbf E\Big[ \sum_{i=\rho_{n,g}}^{\rho_{n,g-1}-1} W_iY_{i,a}'  \mid N_n\Big] \ge (1+o_P(1)) \frac 1{a(a-1)}\frac {nL(n)}{L^*(n)^2} \frac 1h \log \frac hg 
\label{upper}
\end{align}
and 
\begin{align}
\mathbf E\Big[ \sum_{i=\rho_{n,g}}^{\rho_{n,g-1}-1} W_iY_{i,a}'  \mid N_n\Big] \le (1+o_P(1)) \frac 1{a(a-1)}\frac {nL(n)}{L^*(n)^2} \frac 1h \log \frac h{g-1} \ .
\label{lower}
\end{align}

(iv) Now we estimate the difference between $\sum_{i=\rho_{n,g}}^{\rho_{n,g-1}-1} W_iY_{i,a}' $ and its conditional expectation given $N_n$. We have
\begin{align*}
\mathbf E\Big[ &\Big( \sum_{i=\rho_{n,g}}^{\rho_{n,g-1}-1} W_iY_{i,a}'   -\mathbf E\Big[ \sum_{i=\rho_{n,g}}^{\rho_{n,g-1}-1} W_iY_{i,a}'  \mid N_n\Big]\Big)^2 \mid N_n\Big]\\
&= \mathbf E\Big[ \Big( \sum_{i=\rho_{n,g}}^{\rho_{n,g-1}-1} W_i\sum_{k=1}^{i} \sum_{A}(I_{E_{i,k,A}}- \mathbf P(E_{i,k,A}\mid N_n) \Big)^2\mid N_n \Big] \\
&= \sum_{i,j=\rho_{n,g}}^{\rho_{n,g-1}-1} W_iW_j\sum_{k=1}^{i}\sum_{l=1}^{j}\sum_{A,A'}(\mathbf P(E_{i,k,A}\cap E_{j,l,A'}\mid N_n)- \mathbf P(E_{i,k,A}\mid N_n)\mathbf P( E_{j,l,A'}\mid N_n)) \text{ a.s.}
\end{align*}
For $A=A'$, $k=l$ and  $i\le j$ we use the estimate
 \begin{align*} \mathbf P(E_{i,k,A}\cap E_{j,k,A}\mid N_n) \le \mathbf P( E_{i,k,A}\mid N_n)  \text{ a.s.}
\end{align*}
Taking also account of Lemma \ref{Lmm5}, we  obtain with some $c>0$
\begin{align*}
\mathbf E\Big[ &\Big( \sum_{i=\rho_{n,g}}^{\rho_{n,g-1}-1} W_iY_{i,a}'   -\mathbf E\Big[ \sum_{i=\rho_{n,g}}^{\rho_{n,g-1}-1} W_iY_{i,a}'  \mid N_n\Big]\Big)^2 \mid N_n\Big]\\
&\le 2 \sum_{i=\rho_{n,g}}^{\rho_{n,g-1}-1} \sum_{j=i}^{\rho_{n,g-1}-1} \sum_{k=1}^{i}\sum_{A}W_iW_j\mathbf P(E_{i,k,A}\mid N_n)\\
&\quad \mbox{}+ c\sum_{i,j=\rho_{n,g}}^{\rho_{n,g-1}-1}   \sum_{k=1}^{i}\sum_{l=1}^{j}\sum_{A,A'}W_iW_j\Big(\frac{1}{X_k}+ \frac 1{X_l}\Big) \mathbf P(E_{i,k,A}\mid N_n)\mathbf P(E_{j,l,A'}\mid N_n) \\
&\le 2 \sum_{j=0}^{\rho_{n,g-1}-1}W_j\sum_{i=\rho_{n,g}}^{\rho_{n,g-1}-1}  \sum_{k=1}^{i}\sum_{A}W_i\mathbf P(E_{i,k,A}\mid N_n)\\
&\quad \mbox{}+ \frac {2c}{X_{\rho_{n,g-1}-1} } \Big(\sum_{i=\rho_{n,g}}^{\rho_{n,g-1}-1}   \sum_{k=1}^{i}\sum_{A}W_i \mathbf P(E_{i,k,A}\mid N_n)\Big)^2\\
&= 2 \tilde \rho_{n,g-1} \mathbf E\Big[ \sum_{i=\rho_{n,g}}^{\rho_{n,g-1}-1} W_iY_{i,a}'  \mid N_n\Big]+ \frac {2ch}n\Big(\mathbf E\Big[ \sum_{i=\rho_{n,g}}^{\rho_{n,g-1}-1} W_iY_{i,a}'  \mid N_n\Big]\Big)^2 \text{ a.s.}
\end{align*}
Using \eqref{lower}, $L_n=o(L^*(n))$ and the fact that $L^*(n) $ is increasing, this implies for $g\ge 2$ 
\[ \mathbf E\Big[ \Big( \sum_{i=\rho_{n,g}}^{\rho_{n,g-1}-1} W_iY_{i,a}'   -\mathbf E\Big[ \sum_{i=\rho_{n,g}}^{\rho_{n,g-1}-1} W_iY_{i,a}'  \mid N_n\Big]\Big)^2 \mid N_n\Big] = o(n) \]
and thus from \eqref{upper} for $g \ge 2$ 
\begin{align*}
\sum_{i=\rho_{n,g}}^{\rho_{n,g-1}-1} W_iY_{i,a}' \ge (1+o_P(1)) \frac 1{a(a-1)}\frac {nL(n)}{L^*(n)^2} \frac 1h \log \frac hg \ .
\end{align*}

(v) The last formula implies for $a \ge 2$ that
\begin{align*}
\widehat \ell_{n,a}&=\sum_{i=0}^{\tau_n-1} W_iY_{i,a} \ge \sum_{i= \rho_{n,h-1}}^{\rho_{n, 1}-1} W_iY_{i,a}'\\&\ge (1+o_P(1)) \frac 1{a(a-1)}\frac {nL(n)}{L^*(n)^2} \sum_{g=2}^{h-1} \frac 1h \log \frac hg \\
&\ge (1+o_P(1)) \frac 1{a(a-1)} \frac {nL(n)}{L^*(n)^2} \int_{2/h}^1 \log \frac 1z \, dz \ .
\end{align*}
Letting $h\to \infty$, we obtain the lower estimate
\[\widehat \ell_{n,a} \ge (1+o_P(1)) \frac 1{a(a-1)} \frac {nL(n)}{L^*(n)^2} \ . \]
For an upper estimate note that $\widecheck \ell_n= \sum_{a \ge 2} \widehat \ell_{n,a}$. This formula and Theorem \ref{Thm3} imply for any natural number $r$
\begin{align*}
\widehat \ell_{n,a} &\le \widecheck \ell_n - \sum_{2\le b \le r, b \neq a} \widehat \ell_{b,n} \\&\le (1+o_P(1))\frac {nL(n)}{L^*(n)^2} \Big(1-  \sum_{2\le b \le r, b \neq a}\frac1{(b-1)b}  \Big)\\
&\stackrel P\sim\frac {nL(n)}{L^*(n)^2}\Big(  \frac 1r + \frac 1{(a-1)a} \Big)\ .
\end{align*}
Letting  $r\to \infty$ yields the  upper estimates and we obtain altogether
\[ \widehat \ell_{n,a} \stackrel P\sim \frac 1{a(a-1)} \frac {nL(n)}{L^*(n)^2} \ . \]
In order to achieve also $L_1$-convergence we deduce from Theorem \ref{Thm3} that the random variables $\widecheck \ell_n L^*(n)^2/(nL(n))$ form a uniformly integrable sequence. Since $\widehat \ell_{n,a} \le \widecheck \ell_n$, the same holds true for the random variables $\widehat \ell_{n,a} L^*(n)^2/(nL(n))$. Thus also $L_1$-convergence follows. This finishes the proof of Theorem \ref{Thm4}. \qed

\end{document}